\documentclass[reqno]{amsart}
\title{Embedding theory of lattices and its application for $2$-integrable lattices}
\author[]{Qianqian Yang$^\dag$}
\thanks{$^\dag$Q. Yang is partially supported by the National Natural Science Foundation of China (No. 12071454) and China Postdoctoral Science Foundation (No. 2020M671855).}
\address{School of Mathematical Sciences, University of Science and Technology of China, 96 Jinzhai Road, Hefei, 230026, Anhui, PR China}
\email{qqyang91@ustc.edu.cn}
\author[]{Kiyoto Yoshino$^\star$}\thanks{$^\star$K. Yoshino is the Corresponding author}
\address{Tohoku University, Graduate School of Information Sciences, 6-3-09~Aoba, Aramamaki-aza, Aoba-ku, Sendai, Miyagi, 980-8579, Japan}
\email{kiyoto.yosino.r2@dc.tohoku.ac.jp}
\usepackage{amsmath,amsthm,amsfonts,amssymb,enumerate}
\usepackage{longtable,multirow}
\usepackage{geometry}
\numberwithin{equation}{section}

\usepackage{color}

\newtheorem{lem}{Lemma}[section]
\newtheorem{thm}[lem]{Theorem}
\newtheorem{prop}[lem]{Proposition}
\newtheorem{cor}[lem]{Corollary}
\theoremstyle{definition}
\newtheorem{dfn}[lem]{Definition}
\newtheorem{rem}[lem]{Remark}

\newtheorem{example}[lem]{Example}

\newcommand{\R}{\mathbb{R}}
\newcommand{\Q}{\mathbb{Q}}

\newcommand{\Z}{\mathbb{Z}}		
\newcommand{\A}{\mathsf{A}}

\newcommand{\sE}{\mathsf{E}}
\newcommand{\e}{\mathbf{e}}
\newcommand{\sR}{\mathsf{R}}

\newcommand{\ba}{\mathbf{a}}
\newcommand{\bb}{\mathbf{b}}
\newcommand{\bc}{\mathbf{c}}
\newcommand{\be}{\mathbf{e}}
\newcommand{\bs}{\mathbf{s}}
\newcommand{\bt}{\mathbf{t}}
\newcommand{\bu}{\mathbf{u}}
\newcommand{\bv}{\mathbf{v}}
\newcommand{\bx}{\mathbf{x}}
\newcommand{\by}{\mathbf{y}}
\newcommand{\bz}{\mathbf{z}}
\newcommand{\bw}{\mathbf{w}}
\newcommand{\pr}{\rho}

\newcommand{\BB}[4]{\begin{pmatrix} #1 & #2 \\ #3 & #4 \end{pmatrix}}

\DeclareMathOperator{\rank}{rank}

\DeclareMathOperator{\Aut}{Aut}
\DeclareMathOperator{\Sym}{Sym}
\DeclareMathOperator{\supp}{supp}
\DeclareMathOperator{\type}{type}

\DeclareMathOperator{\sign}{sign}

\DeclareMathOperator{\M}{M}

\begin{document}

	\begin{abstract}
	   For a positive integer $s$, a lattice $L$ is said to be $s$-integrable if $\sqrt{s}\cdot L$ is isometric to a sublattice of $\Z^n$ for some integer $n$.
		Conway and Sloane found two minimal non $2$-integrable lattices of rank $12$ and determinant $7$ in 1989.
		We find two more ones of rank~$12$ and determinant $15$.
		Then we introduce a method of embedding a given lattice into a unimodular lattice,
		which plays a key role in proving minimality of non $2$-integrable lattices and finding candidates for non $2$-integrable lattices.
	\end{abstract}
   
	
\maketitle
%	\begin{enumerate}
%		\item I like to write tex codes as a programming language. So I use ``indent''. But you should write as you like.
%		\item Every vector must be written such as $\bf{v}$ by using the command ``bf'' or ``mathbf''.
%		\item Don't use bold to emphasize new notations. The command ``emph'' is recommended by AMS.
%			For example, ``\emph{lattice}'' may be better than ``{\bf lattice}''.
%		\item ``Basic quadratic form'' + ``this paper''+ $\varepsilon$ explain ``Embedding theory''.
%		\item Notation:
%			\begin{enumerate}
%				\item $F$-quad. sp. $\simeq$ $F$-quad. sp.
%				\item $R$-lattice $\simeq$ (or $\simeq_R$) $R$-lattice.
%				\item $R$-lattice $\cong$ $R$-matrix.
%				\item $R$-matrix $\sim_R$ $R$-matrix.
%				\item $F$-quad. space $\cong$ $F$-matrix.
%			\end{enumerate}
%		\item $a,b,\ldots$ : elements of $R$ and $F$ etc.
%		\item All(most) definitions follow the book ``Basis quadratic form''.		
%	\end{enumerate}
\maketitle
\section{Introduction}	\label{sec:intr}
	This paper is related to one of J. H. Conway's results.
	We were very surprised to hear of his untimely death from the virus.
	We mourn it and pay tribute to his greatness.

	In this paper, a lattice we mean is a positive definite integral $\mathbb{Z}$-lattice and a unimodular lattice is a positive definite unimodular $\Z$-lattice, if we do not specify it. Let $s$ be a positive integer.
	A lattice $L$ is said to be \emph{$s$-integrable} if $\sqrt{s}\cdot L$ is isometric to a sublattice of $\Z^n$ for some integer $n$.
	Let $\phi(s)$ be the smallest rank in which there is a non $s$-integrable lattice.
%Previous Research
	The values $\phi(1)=6$, $\phi(2)=12$ and $\phi(3)=14$ were shown in~\cite[Theorem~1]{CS1989},
	and the value $\phi(s)$ is not determined if $s$ is at least $4$.

	A lattice $L$ is said to be \emph{non $s$-minimal},
	if there exist a lattice $M$ and a positive integer $m$ such that $\sqrt{s} \cdot L$ is isometric to a sublattice of $\sqrt{s} \cdot M \perp \Z^m$ which is not contained in $\sqrt{s} \cdot M$.
	Otherwise it is said to be \emph{$s$-minimal}. 
	Notice that a nonzero $s$-integrable lattice is always non $s$-minimal. 
	To erase language, if a non $s$-integrable lattice is $s$-minimal, we say it is a \emph{minimal non $s$-integrable lattice}.
	In the case of $s=1$, Ko~\cite{Ko1939,Ko1942a,Ko1942b} proved that the lattices $\sE_6$, $\sE_7$ and $\sE_8$ are unique minimal non $1$-integrable lattices of rank $6$, $7$ and $8$ respectively,
	and Plesken~\cite{P} gave a short proof by embedding lattices into unimodular lattices.
	Conway and Sloane gave non $2$-integrable lattices as shown in Theorem~\ref{thm:disc7},
	and suspected these lattices are the only minimal non $2$-integrable lattices of rank $12$.

	\begin{dfn}\label{dfn:A15p}
		For each positive integer $n$,
		let
		$
			\A_{n} := \{ \mathbf{x} \in \Z^{n+1} \mid (\mathbf{x},\e) = 0 \}
		$ be a lattice,
		where $\e$ denotes the all one vector in $\Z^{n+1}$.
		Let $\A_{15}^+$ denote the unimodular overlattice of $\A_{15}$,
		that is,
		the lattice generated by $\A_{15}$ and the vector
		\begin{align*}
			[4] := (4,4,4,4,4,4,4,4,4,4,4,4,-12,-12,-12,-12 )/16 \in \R^{16}.
		\end{align*}
	\end{dfn}
	
	\begin{thm}[{\cite[Theorem~14]{CS1989}}]	\label{thm:disc7}
		The orthogonal sublattices in $\A_{15}^{+}$ to a sublattice in $\A_{15}^+$ with Gram matrix
		\begin{align}	\label{thm:disc7:1}
			\begin{pmatrix}
				3 & 2 & 2 \\
				2 & 3 & 2 \\
				2 & 2 & 3
			\end{pmatrix}
		\end{align}
		are non $2$-integrable lattices of rank $12$ and determinant $7$.
	\end{thm}
	
	Furthermore, Conway and Sloane remarked that Theorem~\ref{thm:disc7} gives precisely two minimal non $2$-integrable lattices up to isometry.
	Our motivation comes from verifying the claim that Conway and Sloane suspected and determining the minimal non $s$-integrable lattices.
	As a main result in this paper, we give two more minimal non $2$-integrable lattices:	
	\begin{thm} \label{thm:main}
		There are precisely two lattices with Gram matrix
		\begin{align}	\label{thm:main:1}
			\begin{pmatrix}
				3 & 2 & 0 \\
				2 & 3 & 0 \\
				0 & 0 & 3
			\end{pmatrix}
		\end{align}
		up to $\Aut(\A_{15}^+)$ in $\A_{15}^+$,
		and they are given by $\langle \ba,\bb,\bc \rangle$ and $\langle \ba,\bb,\bc' \rangle$,
		where
		\begin{align*}
			&\ba:= ( -3,	-3,-3,	-3,1,		1,1,		1,		1,1,1,1,1,1,1,1 )/4 \in \A_{15}^{+},	\\
			&\bb := ( -3,	-3,-3,	1,-3,		1,1,		1,		1,1,1,1,1,1,1,1 )/4 \in \A_{15}^{+},	\\
			&\bc:= ( -3,	1,1,		1,1,		-3,-3, -3,	1,1,1,1,1,1,1,1 )/4 \in \A_{15}^{+},	\\
			&\bc':= ( 1,		1,1,		-3,-3,	-3,-3	,	1,		1,1,1,1,1,1,1,1 )/4 \in \A_{15}^{+}.
		\end{align*}
		The non-isometric sublattices $\langle \ba,\bb,\bc \rangle^{\perp}$ and $\langle \ba,\bb,\bc' \rangle^{\perp}$
		are minimal non $2$-integrable lattices of rank $12$ and determinant $15$.
	\end{thm}	
	We show an unified way to prove the non $2$-integrability of the lattices given in above two theorems.	
	It is natural to wonder if there exist more minimal non $2$-integrable lattices of rank $12$,
	and this problem is still open.
	
	In order to prove the minimality of non $2$-integrable lattices in Section~\ref{sec:minimality} and explain how to find candidates for non $2$-integrable lattices in Subsection~\ref{subsec:app},
	we introduce a method of embedding lattices into unimodular lattices as follows. For undefined notation, we refer to next section.
	\begin{thm}\label{thm:embedding general}
    	Let $m$ and $n$ be positive integers. Let $L$ be a lattice on the $n$-dimensional quadratic $\Q$-space $V$.
    	Then $L$ is a sublattice of a unimodular lattice of rank $m$ if and only if one of the following holds:
	    \begin{enumerate}
	        \item \label{thm:embedding general:1}
	        	$m=n$, and for each prime number $p$, $\det(V_p) = 1$ and $S_p(V) = 1$.
	    	\item \label{thm:embedding general:2}
	    		$m = n+1$, and for each prime number $p$,
                $
                    S_p(V)(\det(V),\det(V))_p=1.
                $
    		\item \label{thm:embedding general:3}
    			$m = n+2$, and for each prime number $p$,
            	$
            	    S_p(V)=\begin{cases}
            	        1   & \text{ if } p>2 \text{ and } \det(V_p)=-1,\\
            	        -1  & \text{ if } p=2 \text{ and } \det(V_2)=-1.
            	   \end{cases}
            	$
		    \item \label{thm:embedding general:4}
		    	$m \geq n+3$.
	    \end{enumerate}
    \end{thm}
    Conway and Sloane~\cite{CS1989} explained the proof of above theorem in the case of embedding lattices of rank $n$ into a unimodular lattice of rank $n+3$ for every positive integer $n$. Following their explanation, we show the proof of this theorem in more details, and give applications. In addition, we present a theorem to embed lattices into an odd unimodular lattice (see Theorem~\ref{embedding to odd}).

	This paper is organized as follows:
	We introduce notation in Section~\ref{sec:notation},
	and well-known results for quadratic spaces in Section~\ref{sec:quad}.
	In Section~\ref{sec:maximal}, for every prime number $p$, we introduce properties of the maximal $\Z_p$-lattices.
	In Section~\ref{sec:embedding theory}, we show a method of embedding a lattice into another by applying the results in the previous two sections.
	In Section~\ref{sec:pri}, we introduce lemmas for primitive lattices.
	In Section~\ref{sec:eut}, we give necessary and sufficient conditions for a lattice to be $s$-integrable and useful lemmas.
	In Section~\ref{sec:a}, we study the lattice $\A_{15}^+$,	and prove the first statement of Theorem~\ref{thm:main}.
	In Section~\ref{sec:minimality}, we discuss the minimality of non $2$-integrable lattices.
	In Section~\ref{sec:proof}, we separately prove the non $2$-integrability and the minimality claimed in Theorem~\ref{thm:main}.

\section{Notation} \label{sec:notation}
%	\TR{
%	\begin{enumerate}
%		\item $R$ is always a p.i.d. and $F = \Frac{R}$.
%		\item Our notations follow ``Basic Quadratic Forms''.
%		\item I suppose that all readers know ``The invariant factors of a module over a principal ideal domain (PID)''. OK??
%		\item I wrote the definition of Hesse symbols because some books call it differently.
%			If necessary, we may add definitions of Hilbert symbol and others.
%		\item $\otimes$ means $\otimes_\Z$.
%	\end{enumerate}
%	}
	We will follow the book~\cite{basic quadratic form} and give the basic notation as follows.
	Throughout this paper, let $R$ denote a principal ideal domain with quotient field $F \supsetneq R$. Let $R^*$ denote the set of units of $R$, and $F^*$ denote the set of the nonzero elements of $F$.

	Let $(V,B,q)$ be a \emph{quadratic $F$-space},
	where $B$ is a symmetric bilinear form on $V$, and $q$ is the quadratic form associated with $B$.
	For simplicity we usually just write $V$.
	We write $V \simeq W$ if two quadratic $F$-spaces $V$ and $W$ are isometric.
	The quadratic spaces mentioned in this paper are always \emph{regular}, that is, they have no nonzero vector $\bv$ such that $B(\bv,\bu) = 0$ holds for all its vector $\bu$.
	Let $\det(V)$ denote the \emph{determinant} of $V$, which is the coset in $F^*/(F^*)^2$ represented by the determinant of the Gram matrix with respect to a basis of $V$.
	
	An $R$-module $L\subseteq V$ is called an \emph{$R$-lattice} in $V$ if $L=0$ or if there exist linearly independent elements $\bv_1,\ldots,\bv_r$ of $V$ such that $L=R\bv_1\oplus\cdots\oplus R\bv_r$.
	We call $\bv_1,\ldots,\bv_r$ a basis of $L$ and $r$ the \emph{rank} of $L$ (and $\rank 0=0$). We say $L$ is \emph{on} $V$ if $\dim(V)=r$.
	We write $L \simeq M$ or $L \simeq_R M$ if two $R$-lattices $L$ and $M$ are isometric.
	Let $\det(L)$ denote the \emph{determinant} of an $R$-lattice $L$, which is the coset in $F^*/(R^*)^2$ represented by the Gram matrix with respect to a basis of $L$. For $a \in F$, let $a L$ denote the $R$-lattice $\{a \bu \mid \bu \in L\}$.

    Let $L'$ be a sublattice of $L$. The \emph{orthogonal complement} of $L'$ in $L$ is the $R$-module $\{\mathbf{u}\in L\mid B(\mathbf{u},\mathbf{v})=0\text{ for all }\mathbf{v}\in L'\}$, which is also a sublattice of $L$ and is denoted by $(L')^{\perp}$.

	For every positive integer $n$, a matrix in $\M_n(R)$ is said to be \emph{unimodular} if its determinant is in $R^*$. The set of unimodular lattices in $\M_n(R)$ is denoted by $GL_n(R)$.	
	For two matrices $M_1$ and $M_2$ in $\M_n(R)$, we say that they are \emph{$R$-congruent}, denote by $M_1 \sim_R M_2$, if there exists a unimodular matrix $P \in GL_n(R)$ such that $P^\top M_1 P = M_2$.
	Given a symmetric matrix $M$ and an $R$-lattice $L$ (resp.~quadratic $F$-space $V$),
	we write $L\cong M$ (resp.~$V \cong M$) if the Gram matrix of $L$ (resp.~$V$) with respect to some basis is $M$. 	
	Furthermore, an $R$-lattice $L$ of rank $n$ is said to be \emph{unimodular} if $L \cong M$ for some symmetric matrix $M \in GL_n(R)$.
	
	In the whole paper, Let $S$ be the set of prime numbers.
	For each $p \in S$, let $\Z_p$ denote the ring of $p$-adic integers,
	$\Q_p$ the field of $p$-adic numbers,
	$\nu_p(a)$ the $p$-adic order of each $a \in \Q_p$,
	$|\cdot|_p$ the $p$-adic valuation,
	and the symbol $(\cdot,\cdot)_p$ the \emph{Hilbert symbol} over $\Q_p$.
	The set $\R$ of real numbers is denoted by $\Q_{\infty}$.
	For each odd prime number $p$,
	let $\delta_p$ denote one of non-square elements of $\Z_p^*$.
	Note that $\{ 1, \delta_p \}$ is a complete system of representatives of $\Z_p^*/(\Z_p^*)^2$.

	Let $L$ be a $\Z$-lattice on the $n$-dimensional $\Q$-quadratic space $V$;
	say $L:=\bigoplus_{i=1}^n\mathbb{Z}\mathbf{v}_i$.
	For each $p \in S \cup \{\infty\}$,
	we define the \emph{localization} $V_p$ of $V$ at $p$ to be the quadratic $\Q_p$-space $V\otimes\mathbb{Q}_p$.
	Moreover, we define the \emph{localization} $L_p$ of $L$ at $p$ to be the $\mathbb{Z}_p$-lattice on $V_p$ generated by $L$, that is,
	\[L_p=\bigoplus_{i=1}^n\mathbb{Z}_p\mathbf{v}_i.\]	
	In addition, for an orthogonal basis $(\bu_1,\ldots,\bu_n)$ of $V:=(V,B,q)$,
	the \emph{Hasse symbol} of $V$ and that of $L$ at $p$ are defined to be
	\begin{align*}
		S_p(V) = S_p(L) := \prod_{1 \leq i < j \leq n} (q(\bu_i),q(\bu_j))_p \in \{-1,1\}.
	\end{align*}
	The signature of $L$ (resp.~$V$) is denoted by $\sign(L)$ (resp.~$\sign(V)$),
	and $L$ (resp.~$V$) is said to be \emph{positive definite} if
	$\sign(L) = n$ (resp.\ $\sign(V)=n$).
	
	Let $L$ be an $R$-lattice on quadratic $F$-space.
	Then the $R$-module
	$
		sL:= \{ B(\bv,\bu) \mid \bv , \bu \in L \}
	$
	is called the \emph{scalar ideal} of $L$,
	and the $R$-modular $nL$ generalized by
	$
		\{ q(\bv) \mid \bv \in L \}
	$
	is called the \emph{norm ideal} of $L$.
	Note that $2(sL) \subseteq nL \subseteq sL$ and $nL = sL$ if $2 \in R^*$.
	
	A $\Z$-lattice $L$ is said to be \emph{integral} if $sL \subseteq\mathbb{Z}$.
	Moreover, an integral $\Z$-lattice $L$ is said to be \emph{even} if $nL\subseteq 2\mathbb{Z}$, otherwise \emph{odd}. Note that every positive definite integral $\Z$-lattice is isometric to a positive definite integral $\Z$-lattice in $\R^n$ equipped with the canonical bilinear form for some positive integer $n$.
	 For simplicity, we call a positive definite integral $\Z$-lattice equipped with the canonical bilinear form \emph{lattice}.
	
\section{Quadratic spaces}	\label{sec:quad}
In this section we introduce fundamental results for quadratic $\Q$-spaces and quadratic $\Q_p$-spaces.

\begin{thm}[{\cite[Theorem~4.29]{basic quadratic form}}]\label{thm:qp space eq condition}
    Let $p$ be a prime number.
    Two quadratic $\Q_p$-spaces $V$ and $W$ are isometric
    if and only if
    \[\dim(V)=\dim(W),~\det(V)=\det(W) \text{  and  } S_p(V)=S_p(W).\]
\end{thm}

\begin{thm}[{\cite[Theorem 4.32]{basic quadratic form}}]\label{thm:qp space existence condition}
	Let $p$ be a prime number.
    Then there exists a quadratic $\mathbb{Q}_p$-space $V$
    if and only if
    \begin{align}\label{thm:qp space existence condition:1}
        (\dim(V), S_p(V))\neq(1,-1) \text{ or } (\dim(V), \det(V), S_p(V))\neq(2,-1,-1).
    \end{align}
\end{thm}

For every $a,b \in \Q^*$, $(a,b)_p=1$ for almost all $p$, that is, there is a finite set $S'$ such that $(a,b)_p=1$ for every $p    \in S \setminus S'$.
Moreover,
\begin{equation}\label{eq:hasse}
    \prod_{p\in S\cup\{\infty\}}(a,b)_p=1
\end{equation}
holds~(see~\cite[Theorem 5.2]{basic quadratic form}).
This immediately implies the following lemma:

\begin{lem}[{\cite[Corollary 5.3]{basic quadratic form}}]\label{hasse symbol}
	$\prod_{p\in S\cup\{\infty\}}S_p(V_p)=1$
	for every nonzero quadratic $\mathbb{Q}$-space $V$.
\end{lem}

\begin{thm}[{\cite[Corollary 5.9]{basic quadratic form}}]\label{q eq}
    Let $V$ and $W$ be two quadratic $\mathbb{Q}$-spaces.
    Then $V \simeq W$
    if and only if
    $V_\infty \simeq W_\infty$
    and $V_p \simeq W_p$ for each prime number $p$.
\end{thm}

\begin{thm}[{\cite[Theorem 1.3]{cassels}}]\label{thm:existence of q space}
    Let $n\geq2$ and $d\in\mathbb{Q}^*$.
    For each $p\in S\cup\{\infty\}$, let $V_{(p)}$ be an $n$-dimensional quadratic  $\mathbb{Q}_p$-space and suppose that
    \begin{enumerate}
        \item \label{thm:existence of q space:1}
        	$\det(V_{(p)})\in d\mathbb{Q}_p^{*2}$,
        \item \label{thm:existence of q space:2}
        	$\prod_{p\in S\cup\{\infty\}}S_p(V_{(p)})=1$, and $S_p(V_{(p)})=1$ for almost all $p$.
    \end{enumerate}
    Then there exists a quadratic $\mathbb{Q}$-space $V$ with
    $\det(V)=d$,
    $\sign(V)= \sign(V_{(\infty)})$
    and
    $V_p \simeq V_{(p)}$ for each $p\in S$.
\end{thm}

\section{Maximality and existence of $\Z_p$-lattices}		\label{sec:maximal}
Let $A$ be a fractional $R$-ideal, that is, $A \subseteq F$ is an $R$-module and $a A \subseteq R$ for some $a \in R$.
An $R$-lattice $L$ on a quadratic $F$-space $V$ is \emph{$A^{(n)}$-maximal} (resp. \emph{$A^{(s)}$-maximal}) if $nL\subseteq A$ (resp. $sL\subseteq A$) and for any $R$-lattice $M$ on $V$ containing $L$, $nM\subseteq A$ (resp. $sM\subseteq A$) implies $M=L$.
Note that, for every odd prime number $p$, a $\mathbb{Z}_p$-lattice is $A^{(n)}$-maximal if and only if it is $A^{(s)}$-maximal.
The following lemma shows that $\Z^{(s)}$-maximal $\Z$-lattices and $\Z_p^{(s)}$-maximal $\Z_p$-lattices are closely related.
\begin{lem}[{\cite[Lemma~9.8]{basic quadratic form}}]\label{local-global max}
	Let $L$ be a $\mathbb{Z}$-lattice on a quadratic $\mathbb{Q}$-space $V$, and let $A$ be a fractional $\mathbb{Z}$-ideal.
	Then $L$ is $A^{(s)}$-maximal (resp.~$A^{(n)}$-maximal) if and only if $L_p$ is $A^{(s)}_p$-maximal (resp.~$A^{(n)}_p$-maximal) for each prime number $p$.
\end{lem}

\subsection{The isometry classes of $\Z_p^{(s)}$-maximal $\Z_p$-lattices for an odd prime number $p$}

The following theorem immediately gives the $\Z_p^{(s)}$-maximal $\Z_p$-lattices up to isometric.

\begin{thm}[{\cite[Theorem 8.8]{basic quadratic form}}]\label{unique maximal}
    Let $F$ be a field with a complete discrete valuation $|\quad|$, and let $R$ be the associated valuation ring. Suppose $V$ is a (regular) quadratic $F$-space and $A$ is a fractional $R$-ideal. Then there is only one isometry class of $A^{(n)}$-maximal $R$-lattices on $V$.
\end{thm}

We set $F:=\Q_p$, $R:=\Z_p$ and $|\cdot| := |\cdot|_p$ for each odd prime number $p$, and then apply this theorem to $\Z_p$-lattices.
Since the existence of quadratic $\Q_p$-spaces are asserted in Theorem~\ref{thm:qp space existence condition}, we derive the following proposition.

\begin{prop}[{\cite[Theorem~4 for odd prime number $p$]{CS1989}}]  \label{prop:zp-maximal}
    Given an odd prime number $p$, there exists a unique $\Z_p^{(s)}$-maximal $\Z_p$-lattice on a quadratic $\Q_p$-space $V$
    if and only if
    condition~\eqref{thm:qp space existence condition:1} is satisfied.
\end{prop}

\begin{example}\label{de: max zp lattice}
	Let $p$ be an odd prime number and $n$ a positive integer.
	A complete system of representatives of isometry classes of $\Z_p^{(s)}$-maximal $\mathbb{Z}_p$-lattices are enumerated by
	$\Z_p$-lattices $H^p_{n,d,\varepsilon}$ of rank $n$, determinant $d$ and Hasse symnol $\varepsilon$ defined as follows:
	\begin{gather*}
		\begin{aligned}
			&H^p_{n,1,1}\cong I_n,
			&&H^p_{n,\delta_p,1}\cong I_{n-1}\oplus(\delta_p),\\
			&H^p_{n,p,1}\cong I_{n-1}\oplus(p),
			&&H^p_{n,p \delta_p, 1}\cong I_{n-1}\oplus(p\delta_p),\\
			&H^p_{n,p \delta_p, -1}\cong I_{n-2}\oplus(\delta_p)\oplus(p)~(n \geq 2),
			&& H^p_{n,p,-1}\cong I_{n-2}\oplus(\delta_p)\oplus(p\delta_p)~(n \geq 2),\\
		\end{aligned}
		\\
		\begin{aligned}
			H^p_{n,-\delta_p,-1}
			&\cong
			\begin{cases}
				I_{n-2}\oplus(p)\oplus(p\delta_p)& \text{ if } p\equiv1\pmod4,\\
				I_{n-2}\oplus(p)\oplus(p) & \text{ if } p\equiv3\pmod4,
			\end{cases}
			~(n \geq 2),	\\
			H^p_{n,-1,-1}
			&\cong
			\begin{cases}
				I_{n-3}\oplus(\delta_p)\oplus(p)\oplus(p\delta_p)& \text{ if } p\equiv1\pmod4,\\
				I_{n-3}\oplus(\delta_p)\oplus(p)\oplus(p) & \text{ if } p\equiv3\pmod4,
			\end{cases}
			~(n \geq 3).
		\end{aligned}
	\end{gather*}
\end{example}

\subsection{Maximality and existence of $\Z_2$-lattices}
In this subsection we introduce fundamental results for $\Z_2$-lattices. First, adopting a similar way as the proof of \cite[PROPOSITION 2]{CS1988}, we can obtain the following result:
\begin{prop}[{Cf.~\cite[PROPOSITION 2]{CS1988}}]	\label{prop:max z2 lattice det}
	If a $\mathbb{Z}_2$-lattice $L$ is $\mathbb{Z}_2^{(s)}$-maximal, then $\nu_2(\det(L))=0$ or $1$.
\end{prop}

\begin{prop}	\label{prop:det Z2}
	There exists a $\Z_2$-lattice whose norm ideal is $\Z_2$ on a quadratic $\Q_p$-space $V$
   if and only if
   condition~\eqref{thm:qp space existence condition:1} and $(\dim(V),\det(V),S_2(V)) \neq (2,3,-1)$ are satisfied.
\end{prop}
\begin{proof}
	We show the necessity by enumerating $\Z_2$-lattices $H_{n,d,\varepsilon}$ of rank $n$, determinant $d$ and Hasse symnol $\varepsilon$ as follows:
   \begin{align*}
		&H_{n,1,1}\cong I_n,
		&& H_{n,-1,1}\cong I_{n-1}\oplus(-1),
		\displaybreak[0]\\
		&H_{n,3,1}\cong I_{n-1}\oplus(3),
		&& H_{n,-3,1}\cong I_{n-1}\oplus(-3),
		\displaybreak[0]\\
		&H_{n,1,-1}\cong I_{n-2}\oplus(-I_2)~(n\geq2),
		&&H_{n,-1,-1}\cong I_{n-3}\oplus(- I_3)~(n\geq3),
		\displaybreak[0]\\
		&H_{n,3,-1}\cong I_{n-3}\oplus(3 I_3)~(n\geq3),
		&&H_{n,-3,-1}\cong I_{n-2}\oplus(-1)\oplus(3)~(n\geq2),
		\displaybreak[0]\\
		&H_{n,2,1}\cong I_{n-1}\oplus(2),
		&& H_{n,-2,1}\cong I_{n-1}\oplus(-2),
		\displaybreak[0]\\
		&H_{n,6,1}\cong I_{n-1}\oplus(6),
		&&H_{n,-6,1}\cong I_{n-1}\oplus(-6),
		\displaybreak[0]\\
		&H_{n,2,-1}\cong I_{n-2}\oplus(-3)\oplus(-6)~(n \geq 2),
		&& H_{n,-2,-1}\cong I_{n-2}\oplus(-3)\oplus(6)~(n \geq 2),
		\displaybreak[0]\\
		&H_{n,6,-1}\cong I_{n-2}\oplus(-3)\oplus(-2)~(n \geq 2),
		&& H_{n,-6,-1}\cong I_{n-2}\oplus(-3)\oplus(2)~(n \geq 2).
	\end{align*}
	(In fact, they give a complete system of representatives of isometry classes of $\Z_2^{(s)}$-maximal $\mathbb{Z}_2$-lattices with $nL = \Z_2$.)

	Next we show the sufficiency.
	Theorem~\ref{thm:qp space existence condition} asserts that every quadratic $\Q_2$-space satisfies condition~\eqref{thm:qp space existence condition:1}.
	Thus it suffices to show that there is no $\Z_2$-lattice $L$ with $nL = \Z_2$ and $(\dim(V), \det(V), S_2(V)) = (2,3,-1)$,
	where $V = \Q_2 \otimes L$.
	By way of contradiction, we suppose that there is such a $\Z_2$-lattice $L$.
	Since $nL = \Z_2$, we have
	$
		L \cong (x) \oplus (y)
	$
	for some $x \in \Z_2^*$ and $y \in \Z_2$.
	Then
	\[
		-1 = S_2(V) = (x,y)_2 = (x,-xy)_2 = (x,-\det(V))_2 = (x,-3)_2 = 1.
	\]
	This is a contradiction, and the desired result follows.
\end{proof}

Note that a $\Z_2$-lattice $L$ with Gram matrix $(2) \oplus (6)$ satisfies that $\det(L) = 3$ and $S_2(L) = -1$.

\section{Embedding theory}  \label{sec:embedding theory}
One of helpful ways to investigate lattices is embedding a lattice to another well-known lattice.
In this section, we aim to prove Theorem~\ref{thm:embedding general} and Theorem~\ref{embedding to odd} which give conditions for a given lattice to be embedded into a unimodular lattice and an odd unimodular lattice, respectively.
	
\subsection{Hasse symbols of unimodular lattices and unimodular $\Z_p$-lattices}
In this subsection we introduce the Hasse symbols of unimodular lattices and unimodular $\Z_p$-lattices.
\begin{lem}\label{unimodular Zp}
    Let $p$ be an odd prime number.
    Suppose that $L$ is a unimodular $\mathbb{Z}_p$-lattice on the $n$-dimensional quadratic $\mathbb{Q}_p$-space $V$. Then $S_p(V)=1$. In particular, if $\det(L)=1$, then $L\cong I_n$.
\end{lem}

\begin{proof}
    Let $G$ be the Gram matrix  of $L$ with respect to some basis.
    Since $p$ is odd prime, $G$ is $\Z_p$-congruent to a diagonal matrix $D$ in $\M_n(\Z_p)$ (see~\cite[p.~369]{CS1999}).
    Then the diagonal entries of $D$ are units in $\Z_p$ as $\det(G) \in \Z_p^*$.
    Note that $(a,b)_p = 1$ for any $a,b \in \Z_p^*$.
    This implies $\prod_{i < j} (D_{ii}, D_{jj})_p = 1$,
    and thus $S_p(L) = 1$.
    This is the desired result.

   Next, we suppose that $\det(L) = 1$.
   By the previous argument, we have $S_p(L)=1$.
   Since $L$ is a $\Z_p^{(s)}$-maximal $\Z_p$-lattice, Proposition~\ref{prop:zp-maximal} implies that $L\simeq H^p_{n,1,1} \cong I_n$.
\end{proof}

We remark that this lemma can also be proved by a classification of $\Z_p^{(s)}$-maximal $\Z_p$-lattices for each odd prime number $p$ in Example~\ref{de: max zp lattice}.

%\begin{proof}
%We first claim that $L$ is $\Z_p^{(s)}$-maximal. Suppose not, then there exists a $\Z_p$-lattice $M$ with $sM\subseteq\Z_p$ containing $L$ as a proper sublattice. This implies that we can find $a\in\Z_p-\{0\}-\Z_p^*$ such that $a^2\det(M)=\det(L)\in\Z_p^*$ holds by \cite[Theorem~6.13~(i)]{basic quadratic form}, and thus $\det(M)\notin\Z_p$. But this  contradicts the fact $sM\subseteq\Z_p$. Therefore, our claim holds.
%
%Considering that a complete list of pair-wise non isometric $\Z_p^{(s)}$-maximal $\Z_p$-lattices has been given in Example \ref{de: max zp lattice}, we have that $L$ must be isometric to one of these $\Z_p$-lattices as we described there. As $L$ is unimodular, it is not hard to see $L\simeq H_{n,1,1}^p$ or $H_{n,\delta_p,1}^p$ and $S_p(L)=1$. In particular, if $\det(L)=1$, $L\simeq H_{n,1,1}^p\cong I_n$.
%\end{proof}

\begin{lem}\label{lem:necessary condition on V}
	Suppose that there exists a unimodular lattice on the $n$-dimensional quadratic $\mathbb{Q}$-space $V$.
	Then
	$\det(V_p)=1,~S_p(V_p)=1$ for every prime number $p$.
\end{lem}
\begin{proof}
	Let $L$ be a unimodular lattice on $V$.
	Since $\det(L) = 1$, we find that $L_p$ is a unimodular $\mathbb{Z}_p$-lattice on $V_p$ with $\det (L_p)=1$. This implies that
	$
		\det(V_p)=1
	$	
	for each $p \in S$.
	From Lemma~\ref{unimodular Zp}, we find that
	$
		S_p(V_p)=1
	$
	if $p$ is odd.
	Note that $S_{\infty}(V\otimes\mathbb{Q}_\infty)=1$ as $L$ is positive definite. Theorem \ref{hasse symbol} shows that
	$S_2(V_2)=1$.
\end{proof}

\subsection{Embedding a quadratic space}
In order to embed $\Z$-lattices, it is essential to embed quadratic $\Q$-spaces.
In this subsection we aim to prove Proposition~\ref{prop:to embed a vector space special}, which will be used in next subsection.

% This lemma is used in only this subsection.
\begin{lem}\label{det and Hasse symbol}
	Let $p$ be a prime number,
	and $m$ and $n$ positive integers with $m > n$.
    Suppose that
    $V$ is an $n$-dimensional quadratic $\mathbb{Q}_p$-space.
    Then there exists a quadratic $\mathbb{Q}_p$-space $U$ which satisfies
    $
        V \perp U \cong I_m,
    $
    if and only if there exists a quadratic $\mathbb{Q}_p$-space $U$ with
    \[
        (\dim(U),\det(U),S_p(U))=(m-n,\det (V),S_p(V)(\det(V),\det (V))_p).
    \]
\end{lem}

\begin{proof}
	Let $U$ be a quadratic $\Q_p$-space $U$.
	Then $V \perp U \cong I_m$ if and only if
    $
        \dim(V)+\dim(U)=m,~\det(V)\det(U)=1
    $
    and
    $
        S_p(V)S_p(U)(\det(V),\det(U))_p=1
	$
    by Theorem \ref{thm:qp space eq condition}.
    This implies the desired result.
\end{proof}

% This lemma is used in only this subsection
\begin{lem}\label{to embed a qp vector space}
    Let $p$ be a prime number,
    and $m$ and $n$ positive integers with $m \geq n$.
    Suppose that $V$ is an $n$-dimensional quadratic $\mathbb{Q}_p$-space.
    Then there exists an $(m-n)$-dimensional quadratic $\mathbb{Q}_p$-space $U$ such that
    $V \perp U \cong I_m$ if and only if one of the following is satisfied:
    \begin{enumerate}
    	\item $m=n$,  $\det(V) = 1$ and $S_p(V) = 1$.
        \item $m=n+1$, and
            $S_p(V)(\det(V),\det(V))_p=1.$
        \item
            $m = n+2$, and
            $
            	S_p(V) =\begin{cases}
                    1   & \text{ if } p>2 \text{ and } \det(V)=-1,\\
                    -1  & \text{ if } p=2 \text{ and } \det(V)=-1.
               \end{cases}
        	$
        \item $m \geq n + 3$.
    \end{enumerate}
\end{lem}
\begin{proof}
This follows from Theorem \ref{thm:qp space existence condition} and Lemma~\ref{det and Hasse symbol} immediately.
\end{proof}

% This lemma is used in only this subsection
\begin{lem}\label{lem:to embed a vector space general}
	Let $m$ and $n$ be positive integers with $m \geq n$.
   Suppose that $V$ is a $n$-dimensional quadratic $\mathbb{Q}$-space with $\sign(V)=n$.
   Then there exists an $(m-n)$-dimensional quadratic $\mathbb{Q}$-space $U$ such that
   \begin{equation}\label{global embedding}
       V \perp U\cong I_m
   \end{equation}if and only if for each prime number $p$, there exists a quadratic $\mathbb{Q}_p$-space $U_{(p)}$ such that
    \begin{equation}\label{local embedding}
        V_p \perp U_{(p)}\cong I_m.
    \end{equation}
\end{lem}
\begin{proof}
	If $m=n$, then Theorem~\ref{q eq} implies the desired result.
	Hence we assume that $m > n$.
    Only the sufficiency needs to be proved. Suppose that there exists a quadratic $\mathbb{Q}_p$-space $U_{(p)}$ such that \eqref{local embedding} holds for each prime number $p$.
    Then by Lemma~\ref{det and Hasse symbol},
    we find that
    $
        \det(U_{(p)})=\det(V_p)\in \det(V)\mathbb{Q}_p^{*2}
    $
    and
    \[S_p(U_{(p)})=S_p(V)(\det(V),\det(V))_p.\]
    Define $U_{(\infty)}\cong I_{m-n}$.
    Then $\det(U_{(\infty)})=1\in\det(V)\mathbb{Q}_\infty^{*2}$ as $V$ is positive definite. Now condition~\eqref{thm:existence of q space:1} of Theorem~\ref{thm:existence of q space} is satisfied. By Lemma \ref{hasse symbol} and \eqref{eq:hasse}, condition~\eqref{thm:existence of q space:2} of Theorem~\ref{thm:existence of q space} is also satisfied
    and thus there exists a quadratic $\mathbb{Q}$-space $U$ with $\sign(U)=m-n$ whose localization $U_p$ is
    isometric to $U_{(p)}$ for each prime number $p$.
    This implies that
    \[(V \perp U)_p=V_p \perp U_p \simeq V_p \perp U_{(p)} \cong I_m \text{ for each prime number }p.\]
    Using Theorem \ref{q eq}, we obtain $V \perp U\cong I_m$.
\end{proof}

Combining Lemma~\ref{to embed a qp vector space} and Lemma~\ref{lem:to embed a vector space general}, we have the following:
\begin{prop}\label{prop:to embed a vector space special}
    Let $m$ be a positive integer greater than $n$.
    Suppose that $V$ is an $n$-dimensional quadratic $\Q$-space with $\sign(V)=n$.
    Then there exists an $(m-n)$-dimensional quadratic $\Q$-space $U$ such that $V \perp U \cong I_m$
    if and only if
    one of the conditions~\eqref{thm:embedding general:1}--\eqref{thm:embedding general:4} in Theorem~\ref{thm:embedding general} are satisfied.
\end{prop}

\subsection{Embedding a positive definite integral $\mathbb{Z}$-lattice}	\label{subsec:embedding a lattice}
In this subsection we combine the above results, and prove Theorem~\ref{thm:embedding general} and Theorem~\ref{embedding to odd}.
In addition,  Corollary~\ref{cor:det} shows that a lattice can be embedded into a unimodular lattice if its determinant satisfies  certain conditions.

\begin{lem}\label{sufficient condition on V}
	Let $n$ be an integer at least $2$.
    Suppose that $V$ is an $n$-dimensional positive definite quadratic $\mathbb{Q}$-space.
    Then the following are equivalent.
    \begin{enumerate}
        \item There exists a unimodular lattice on $V$.	
            \label{sufficient condition on V:1}
        \item \label{sufficient condition on V:2}
            For each prime number $p$, the localization $V_p$ of $V$ satisfies that	
            $\det(V_p)=1$ and $S_p(V_p)=1$.
        \item Every $\mathbb{Z}^{(s)}$-maximal $\mathbb{Z}$-lattice on $V$ is unimodular.	
            \label{sufficient condition on V:3}
        \item $V\cong I_n$.	
            \label{sufficient condition on V:4}
    \end{enumerate}
\end{lem}
\begin{proof}
    \eqref{sufficient condition on V:1}$\Rightarrow$\eqref{sufficient condition on V:2}:
    This follows from Lemma~\ref{lem:necessary condition on V}.
    \eqref{sufficient condition on V:3}$\Rightarrow$\eqref{sufficient condition on V:1}:
    This clearly holds.
    \eqref{sufficient condition on V:2}$\Leftrightarrow$\eqref{sufficient condition on V:4}:
    This follows from Theorem~\ref{q eq}.
    \eqref{sufficient condition on V:2}$\Rightarrow$\eqref{sufficient condition on V:3}:
    Let $L$ be a $\mathbb{Z}^{(s)}$-maximal $\Z$-lattice on $V$.
    If $p$ is an odd prime number, then we have $V_p \cong I_n$ by Theorem~\ref{thm:qp space eq condition}.
    This together with Proposition~\ref{prop:zp-maximal} implies $L_p \cong I_n$.
    Next we consider the case of $p=2$.
    By Proposition~\ref{prop:max z2 lattice det}, $\nu_2(\det (L_2)) \in \{0,1\}$ follows.
    This together with $\det(V_2)=1$ implies that $\nu_2( \det( L_2 )) = 0$.
    Therefore $\det(L_p)\in\mathbb{Z}^*_p$ for every prime number $p$, and $\det(L)\in\bigcap_{p\in     S}(\det(L_p)\cap\mathbb{Z})=\{1,-1\}$. This shows that $L$ is unimodular.
\end{proof}

	Theorem~\ref{thm:embedding general} follows from Proposition~\ref{prop:to embed a vector space special} and Lemma~\ref{sufficient condition on V}.
	As a corollary of Theorem \ref{thm:embedding general}, we immediately derive the following:

%\begin{cor}\label{det square free}
%Suppose that $L$ is a positive definite integral lattice of rank $n$ and $\det (L)=d$, where for each odd prime number $p$, $\nu_p(d)$ is odd if $\nu_p(d)\neq0$. Then $L$ can be embedded in a positive definite unimodular $\mathbb{Z}$-lattice of rank $n+2$, except that
%\[\nu_2(d)\text{ is even and }d/2^{\nu_2(d)}\equiv-1\pmod8.\]
%\end{cor}
%\begin{proof}
%\end{proof}

\begin{cor}\label{cor:det}
    Suppose that $L$ is a lattice of rank $n$ and     $\det(L)=p_1^{\alpha_1}\cdots p_t^{\alpha_t}d$,
    where $p_1<p_2<\cdots<p_t$ are odd prime numbers,
    $\alpha_1,\ldots,\alpha_t$ are positive even numbers,
    and $d$ is an integer with $\operatorname*{gcd}(d,p_1\cdots p_t)=1$.
    Then $L$ is a sublattice of a unimodular lattice of rank  $n+2$,
    if the following conditions are satisfied:
    \begin{enumerate}
        \item For each odd prime number $p\in S-\{p_1,\ldots,p_t\}\cup\{2\}$, $\nu_p(d)$ is odd if  $\nu_p(d)>0$.
        \item For each odd prime number $p\in \{p_1,\ldots,p_t\}$, the Legendre symbol $\left(\frac{-d}{p}\right)$ equals $-1$.
        \item $d/2^{\nu_2(d)}\not\equiv-1\pmod8$ if $\nu_2(d)$ is even.
    \end{enumerate}
\end{cor}
\begin{proof}
    Let $V=L\otimes\mathbb{Q}$ be the $n$-dimensional quadratic $\mathbb{Q}$-space.
    Let $S'=\{p\in S\mid \nu_p(d)\text{ is odd}\}$. If $p\in S-S'\cup\{2,p_1,\ldots,p_t\}$,
    then $L_p$ is unimodular,
    and hence $S_p(V)=1$ by Lemma~\ref{unimodular Zp}.
    Next we easily show that for each $p\in S'\cup\{2,p_1,\ldots,p_t\}$, $\det(V_p) \neq 1$.
    Following from Theorem \ref{thm:embedding general} (2), we have the result immediately.
\end{proof}

\begin{thm}[{\cite[Theorem 9.4]{basic quadratic form}}]\label{local modify}
    Let $L$ be a $\Z$-lattice on the $n$-dimensional quadratic $\mathbb{Q}$-space $V$.
    Suppose $T$ is a finite subset of $S$, and suppose that for each $p\in T$,
    a $\mathbb{Z}_p$-lattice $M_{(p)}$ is given on $V_p$.
    Then there is a $\Z$-lattice $L'$ on $V$ such that
    \[L'_p=\left\{\begin{array}{ll}
     M_{(p)} & \text{ if }p\in T, \\
     L_p & \text{ if }p\in S-T.
     \end{array}\right.
    \]
\end{thm}

\begin{thm}\label{embedding to odd}
	Let $m$ be a positive integer.
	Suppose that $L$ is a lattice on the $n$-dimensional quadratic $\mathbb{Q}$-space $V$ and one of the conditions \eqref{thm:embedding general:1}--\eqref{thm:embedding general:4} in Theorem~\ref{thm:embedding general} is satisfied. Then $L$ is a sublattice of an odd unimodular lattice of rank $m$ if 
	\begin{enumerate}
		\item $L$ is odd, 
		\item $m=n+2~\mathrm{ and }~(\det(V_2), S_2(V_2)) \neq (3,1),$ or 
		\item $m\geq n+3$.
	\end{enumerate} 
\end{thm}
\begin{proof}
	While $L$ is odd, the desired result holds immediately by Theorem~\ref{thm:embedding general}. So now we may assume $m=n+2$, the condition \eqref{thm:embedding general:3} in Theorem~\ref{thm:embedding general} is satisfied, and $(\det(V_2), S_2(V_2)) \neq (3,1)$. By Proposition~\ref{prop:to embed a vector space special},
	we find that there exists a $2$-dimensional quadratic $\mathbb{Q}$-space $U$ such that
	\[V \perp U\cong I_{n+2}.\]
	Let $N$ be an integral $\mathbb{Z}$-lattice on $U$.
	Since $(\det(U_2),S_2(U_2))\neq(-1,-1)$ by Theorem \ref{thm:qp space existence condition} and \[(\det(U_2),S_2(U_2))=(\det(V_2),S_2(V_2)(\det(V_2),\det(V_2))_2)\neq(3,-1),\]
	 Proposition~\ref{prop:det Z2} implies that there exists a $\mathbb{Z}_2$-lattice $H$ with $nH = \Z_2$ on $U_2$ so that $U_2 \simeq \Q_2 \otimes H$.
	Using Theorem \ref{local modify},
	we find that there exists a $\mathbb{Z}$-lattice $N'$ such that
	\[
	N'_p=\begin{cases}
		H & \text{ if } p=2, \\
		N_p     & \text{ if } p>2.
	\end{cases}
	\]
	Since $(sN')_p =  sN'_p\subseteq\mathbb{Z}_p$ for every prime number $p$,
	we have that $sN'\subseteq\bigcap_{p\in S}(\mathbb{Z}_p\cap\mathbb{Q})=\mathbb{Z}$,
	and hence $N'$ is integral.
	Moreover, $N'$ is odd as $nN'_2=nH = \mathbb{Z}_2$.
	Let
	\[M= L \perp N'. \]
	Note that the integral $\mathbb{Z}$-lattice $M$ is odd and the quadratic $\mathbb{Q}$-space $M\otimes\mathbb{Q}$ is isometric to $I_{n+2}$.
	A $\mathbb{Z}^{(s)}$-maximal $\mathbb{Z}$-lattice on $M\otimes\mathbb{Q}$ which contains $M$ is a desired positive definite odd unimodular $\mathbb{Z}$-lattice by Lemma~\ref{sufficient condition on V}.
	
	If $m \geq n+3$, then the desired result holds in a similar way.
\end{proof}

\subsection{Applications} \label{subsec:app}
In this subsection, we prove Corollary~\ref{cor:disc of non 2-int} which gives a sufficient condition for a lattice of rank $12$ to be $2$-integrable and  Corollary~\ref{cor:no another lattice} which explains how to find candidates for non $2$-integrable lattices of rank $12$.

The unimodular lattices of rank up to $25$ are completely classified (see~\cite[Chapter~16--18]{CS1999}).
Conway and Sloane~\cite{CS1989} studied the $s$-integrability of unimodular lattices among them.
The following is a part of their results:

\begin{thm}[{\cite[Proof of Theorem~12]{CS1989}}]   \label{thm:uni of rank 14}
    Every unimodular lattice of rank up to $14$ is $2$-integrable.
\end{thm}
\begin{thm}[{\cite[Theorem~13]{CS1989}}]	\label{thm:uni of rank 15}
	The lattice $\A_{15}^+$ is a non $2$-integrable unimodular lattice of rank $15$.
\end{thm}

As $A_{15}^+$ is the unique irreducible unimodular lattice of rank $15$ (see \cite[p.~49]{CS1999}), following from  Theorems~\ref{thm:uni of rank 14}, \ref{thm:uni of rank 15} and~\ref{thm:embedding general}, we have:
\begin{lem}	\label{lem:12}
	The lattice $A_{15}^+$ is the unique unimodular lattice of rank $15$ which is not $2$-integrable. In particular, every non $2$-integrable lattice of rank $12$ is a sublattice in $\A_{15}^+$.
\end{lem}

We derive the following as a corollary of Corollary~\ref{cor:det}:
\begin{cor}	\label{cor:disc of non 2-int}
   Suppose that  $L$ is a non $2$-integrable lattice of rank $12$ and determinant at most $27$.
   Then the determinant of $L$ is equal to one of $7,15,18,23$ and $25$.
\end{cor}
\begin{proof}
	Let $L$ be a lattice of rank $12$.
	Suppose that $\det(L)$ is not equal to $7,15,18,23$ and $25$.
	Then Corollary~\ref{cor:det} implies that $L$ is contained in a unimodular lattice of rank $14$.
	This together with Theorem~\ref{thm:uni of rank 14} implies that $L$ is $2$-integrable.
\end{proof}

This implies the following corollary, which gives candidates for non $2$-integrable lattices.

\begin{cor} \label{cor:no another lattice}
	Let $M$ be a sublattice in $\A_{15}^{+}$ which is generated by $3$ linearly independent elements of norm $3$,
	and $L$ be the sublattice orthogonal to $M$ in $\A_{15}^+$.	
	If $L$ is non $2$-integrable, then it is isometric to one of the lattices of rank $12$ in Theorem~\ref{thm:disc7} and Theorem~\ref{thm:main}.
\end{cor}
\begin{proof}
	We enumerate the positive definite matrices whose diagonal entries are $3$ in common and off-diagonal entries are in $\{-2,-1,0,1,2\}$, and let $\mathcal{G}$ be the set of these matrices.
	It is verified that $G$ is $\Z$-congruent to either the matrix~\eqref{thm:disc7:1} or \eqref{thm:main:1} for each $G \in \mathcal{G}$ with $\det(G) \in \{7,15\}$.
	Hence it suffices to show that, if $\det(M) \neq 7,15$, then $L$ is $2$-integrable.
	
	Suppose that $\det(M) \neq 7,15$.	
	Note that the Gram matrix with respect to some basis of $M$ is contained in $\mathcal{G}$.
	By calculating the determinants of all matrices in $\mathcal{G}$, we have
	\[
		\det(M) \in \{3, 5, 7, 8, 12, 13, 15, 16, 20, 21, 24,27\} \setminus \{7,15\}.
	\]
	Set $P := (M\otimes \Q) \cap \A_{15}^+$. Then $P$ is primitive in $\A_{15}^+$.
	(The definition of ``primitive'' is written in Section~\ref{sec:pri}.)
	Lemma~\ref{lem:pri4} implies $\det(L) = \det(P)$ since $L = M^\perp = P^\perp$ in $\A_{15}^+$.
	In addition, we have $\det(M) = \det(P)  [P:M]^2$.
	Hence we see that
	\[
		\det(L) = \det(P) \in \{1,2,3,4,5,6,7,8,9,12,13,15,16,20,21,24,27 \} \setminus \{7,15\}.
	\]
	This together with Corollary~\ref{cor:disc of non 2-int} implies that $L$ is $2$-integrable.
\end{proof}

\begin{rem}
	It is a natural question to wonder if we obtain more candidates for non $2$-integrable lattices of rank $12$ in $\A_{15}^+$.
	By using computer, we derive a better result than Corollary~\ref{cor:no another lattice} as follows:
	As discussed in Lemma~\ref{lem:thm}, it is possible to enumerate the lattices in $\A_{15}^+$ each of which is orthogonal to a lattice of rank $3$ generated by $3$ linearly independent elements of norm at most $4$.
	Since we can judge whether a given lattice is $2$-integrable by solving (with computer) the corresponding linear integer programming problem (see Lemma~\ref{lem:int_prog}), it turns out that there is no non $2$-integrable lattice among them except the non $2$-integrable lattices obtained in Theorem~\ref{thm:disc7} and Theorem~\ref{thm:main}.
	Now we may not immediately verify this result without computer.
\end{rem}
	
\section{Primitive sublattices}	\label{sec:pri}
For a lattice $L$, its \emph{dual} is the lattice $\{\mathbf{u}\in L\otimes \Q \mid (\mathbf{u},\mathbf{v}) \in \Z \text{ for all }\mathbf{v}\in L\}$, and we denote it by $L^*$. Let $L$ be a lattice, and $M$ a sublattice of it. 
The lattice $M$ is said to be \emph{primitive} if $M=M^* \cap L$. Note that $M=M^* \cap L$ if and only if $L/M$ is a free $\Z$-module.

\begin{lem}	\label{lem:pri3}
	Let $M$ be a sublattice of a lattice.
	If the determinant of $M$ is square free,
	then $M$ is primitive.
\end{lem}

\begin{proof}
    Suppose that $M$ is a sublattice of a lattice $L$.
    Then we have $M \subseteq M^* \cap L$,
    and hence $\det(M) = \det( M^* \cap L) \cdot [M^* \cap L : M]^2$.
    Since $\det(M)$ is square free, $M^* \cap L = M$ follows.
\end{proof}

\begin{lem}[{\cite[Proposition~1.2]{E}}]	\label{lem:pri4}
	Let $L$ be a unimodular lattice and $M$ be its primitive sublattice.
	Then, the determinant of $M$ is equal to that of the sublattice $M^{\perp}$ orthogonal to $M$ in $L$.	
\end{lem}

\begin{lem}	\label{lem:pri2}
	Let $L$ be a unimodular lattice and $M$ be its primitive sublattice.
	Then,
	\begin{align}
		(M^{\perp})^* \perp M^* = L+M^* = \bigsqcup_{\bu+M \in M^*/M } \left( \bu+L \right), \label{lem:pro2:1}	
	\end{align}
	where $M^\perp$ denotes the sublattice orthogonal to $M$ in $L$.
\end{lem}	

\begin{proof}
	Since $L=L^*$, Lemma~\ref{lem:pri4} implies that
	\begin{align*}
		[ (M^{\perp})^* \perp M^* : L ]
		= [L:M^{\perp} \perp M]
		= \det(M).
	\end{align*}
	In addition,
	since $M$ is primitive,
	we have
	\begin{align*}
		[L + M^* : L ] = |(L + M^*) / L | = |M^* /(M^* \cap L)| = |M^*/M| = \det(M).
	\end{align*}
	Since $L \subseteq L + M^* \subseteq (M^{\perp})^* \perp M^*$,
	these forces $(M^{\perp})^* \perp M^* = L+M^*$.
	Thus, the first equality follows.
	The second equality follows by $M=M^* \cap L$.
\end{proof}

\section{The $s$-integrability and eutactic stars of scale $s$}	\label{sec:eut}
Since it is difficult to determine whether a lattice is $s$-integrable from its definition, Conway and Sloane~\cite{CS1989} gave equivalent conditions for a given lattice to be $s$-integrable in terms of eutactic stars.
Here we introduce them, and consider a eutactic star of scale $2$.
Hereafter, we let $\be_i$ denote the vector of which the $i$-th entry is $1$ and the others are $0$.

\begin{dfn}	\label{dfn:eut}
	Let $s$ be a positive integer.
	For positive integers $m \geq n$,
	let $\pr$ be the orthogonal projection from $\R^m$ to an $n$-dimensional subspace.
	Then vectors $\pr(\sqrt{s} \cdot \be_1) , \ldots, \pr(\sqrt{s} \cdot \be_m)$ (with repetitions allowed)
	are said to form an \emph{($n$-dimensional) eutactic star (of scale $s$)}.
\end{dfn}

\subsection{Eutactic stars of scale $s$}
Most proofs of non $s$-integrability of a given lattice are reduced to arguments using the following theorem and lemma.

\begin{thm}[{\cite[Theorem 3]{CS1989}}]	\label{thm:iff_eut}
	Let $s$ be a positive integer.
	A lattice $L$ of rank $n$ is $s$-integrable if and only if
	its dual $L^*$ contains an $n$-dimensional eutactic star of scale $s$.
\end{thm}

\begin{lem}[{\cite[pp.~215--216]{CS1989}}]\label{lem:eut}
	A necessary and sufficient condition for $\bs_1,\ldots,\bs_m \in \R^n$ to be an $n$-dimensional eutactic star of scale $s$
	is that, for each $\bw \in \R^n$,
	\begin{align}\label{lem:eut:1}
		\sum_{i = 1}^m (\bw,\bs_i)^2 = s (\bw,\bw).
	\end{align}			
\end{lem}

According to the following lemma,
determining whether a lattice is $s$-integrable is equivalent to judging the existence of a non-negative integer solution of a system of linear equations.
Hence, it can be determined by computer if the number of variables is few.

\begin{lem}	\label{lem:int_prog}
	Let $s$ be a positive integer,
	$L$ a lattice with a basis $\bw_1,\ldots,\bw_n$,
	and $\bu_1,\ldots,\bu_N$ the pairwise distinct vectors in $L^*$ of norm at most $s$.
	Then $L$ is $s$-integrable if and only if the following system of equations has a non-negative integer solution $(x_1,\ldots,x_N)$:
	\begin{align}	\label{lem:int_prog:1}
		\sum_{k=1}^N (\bw_i+\bw_j,\bu_k)^2 x_k = s (\bw_i+\bw_j,\bw_i+\bw_j) \quad (i,j = 1,\ldots,n).
	\end{align}
\end{lem}

%\begin{proof}
%		Theorem~\ref{thm:iff_eut} asserts that
%		$L$ is $s$-integrable
%		if and only if
%		$L^*$ contains an $n$-dimensional eutactic star of scale $s$.
%		By applying Lemma~\ref{lem:eut} to $\R L$,
%		$L^*$ contains such a eutactic star			
%		if and only if there are $\bs_1,\ldots,\bs_m$ in $L^*$
%		(of norm at most $s$ with repetitions allowed)
%		such that the equality~(\ref{lem:eut:1}) holds for every $\bw \in \R L$.
%		Moreover,
%		it holds
%		if and only if the following holds:
%		\begin{align}	\label{lem:int_prog:2}
%			\sum_{k=1}^m (\bw_i+\bw_j,\bs_k)^2 = s (\bw_i+\bw_j,\bw_i+\bw_j) \quad (i,j = 1,\ldots,n).
%		\end{align}
%		Since such $\bs_1,\ldots, \bs_m$ belong to the set $\{\bu_1, \ldots, \bu_N\}$,
%		the existence of $\bs_1,\ldots,\bs_m$ in $L^*$ of norm at most $s$ satisfying the quality~(\ref{lem:int_prog:2})
%		is equivalent to that of a non-negative integer solution $(x_1,\ldots,x_N)$ of the equation~(\ref{lem:int_prog:1}).		
%		More specifically, if such $\bs_1,\ldots,\bs_m$ exist,
%		then a solution $(x_1,\ldots,x_N)$ of the equation~(\ref{lem:int_prog:1}) is given by letting $x_j = |\{i \mid \bs_i = \bu_j \}|$ for $j = 1,\ldots, N$.
%\end{proof}
		
\begin{proof}
	Theorem~\ref{thm:iff_eut} asserts that $L$ is $s$-integrable if and only if $L^*$ contains an $n$-dimensional eutactic star of scale $s$.
	Thus it is sufficient to show that the two conditions that the dual lattice $L^*$ contains an $n$-dimensional eutactic star of scale $s$ and that the equation~(\ref{lem:int_prog:1}) has a non-negative integer solution are equivalent.

	Suppose that $\bs_1,\ldots, \bs_m$ in $L^*$ is a eutactic star of scale $s$.
	As $(\bs_i,\bs_i)\leq s$ for each $i$, we have $\bs_1,\ldots, \bs_m\in\{\bu_1, \ldots, \bu_N\}$.
	Now applying Lemma~\ref{lem:eut}, we find that a solution $(x_1,\ldots,x_N)$ of the equation~(\ref{lem:int_prog:1}) can be given by setting $x_j = |\{i \mid \bs_i = \bu_j \}|$ for $j = 1,\ldots, N$.
	
	Now suppose $(x_1,\ldots,x_N)$ is a non-negative integer solution for the equation~(\ref{lem:int_prog:1}). Then the multiple set $\{\bu_1^{(x_1)},\ldots,\bu_N^{(x_N)}\}$ is a eutactic star of scale $s$ in $L^*$ by using Lemma~\ref{lem:eut} again. This completes the proof.
\end{proof}

\begin{lem}	\label{lem:hoge}
	Let $L$ be a lattice,
	$N$ its sublattice, and
	$\bw$ a nonzero element in $N$.
 If $\bs_1,\ldots,\bs_m\in N^* \setminus \Q \bw$ form a eutactic star of scale $s$,
	then
	\begin{align*}
        s(\bw,\bw) \leq
        |\{ i \in \{1,\ldots,m\} \mid (\bw,\bs_i) \neq 0 \}| \lceil \sqrt{s(\bw,\bw)}-1 \rceil^2.
	\end{align*}
\end{lem}		
\begin{proof}
	Since $\bs_i \not\in \Q \bw$ and $\bw \in N$, for the integer $|(\bw,\bs_i)|$, we have
	 $$|(\bw,\bs_i)|<\sqrt{(\bw,\bw)(\bs_i,\bs_i)}\leq\sqrt{s(\bw,\bw)}.$$
	This together with Lemma~\ref{lem:eut} implies the desired conclusion.
\end{proof}

\subsection{Eutactic stars of scale $2$}
Assume that vectors $\bs_1,\ldots,\bs_m$ form a eutactic star of scale $s$, and $G$ is the Gram matrix of them.
By Definition \ref{dfn:eut}, we conclude that the matrix $sI-G$ is positive semi-definite.
Using this, we give lemmas to examine the properties of each pair of $\bs_i$ and $\bs_j$.	

\begin{lem}	\label{lem:gram}
	For two real symmetric matrices $G := \BB{\alpha}{\beta}{\beta}{\gamma}$ and $A := \BB{2}{\delta}{\delta}{2}$, the following hold:
	\begin{enumerate}
		\item	If $G$ is positive semi-definite,		\label{c1}
			then	$\alpha \gamma \geq \beta^2$.
		\item		\label{c2}
			If both $2I-G$ and $A-G$ are positive semi-definite,
			then $(2-\alpha)(2-\gamma) \geq (\delta/2)^2.$
	\end{enumerate}		
\end{lem}
\begin{proof}
	(\ref{c1})
	The desired result clearly holds.
	(\ref{c2})
	Since $2I-G$ is positive semi-definite, we have
	\[
		(2-\alpha)(2-\gamma) \geq \beta^2 = ( \delta/2 + (\beta - \delta/2) )^2.
	\]
	Similarly,
	since $A-G$ is positive semi-definite,
	we have
	\[
		(2-\alpha)(2-\gamma) \geq (\delta-\beta)^2 = ( \delta/2 - (\beta - \delta/2) )^2.
	\]
	Thus, adding the above two inequalities, we derive the desired inequality.
%	$
%		2(2-\alpha)(2-\gamma) \geq 2(\delta/2)^2 + 2(\beta - \delta/2) )^2	\geq 2(\delta/2)^2.
%	$
\end{proof}

The following lemma helps to restrict candidates for eutactic stars.

\begin{lem}	\label{lem:v}
	Let $L$ be a lattice,
	and $\bv_1, \ldots, \bv_m$ elements of norm $2$ in $L$.
	Let $N$ be a sublattice of $L$,
	and $\bs_i$ denote the orthogonal projections of $\bv_i$ to $N\otimes\Q$ for each $i$.
	Suppose that $\bs_1,\ldots,\bs_m$ form a eutactic star of scale $2$.
	Then the following hold:
	\begin{enumerate}
		\item	\label{lem:v:1}
		If the norms of $\bs_1,\ldots,\bs_m$ are greater than $1$,	then $\bv_1,-\bv_1, \ldots,\bv_m, -\bv_m$ are pairwise distinct.
		\item	\label{lem:v:2}
		For any integers $i$ and $j$,
		if $(2-(\bs_i,\bs_i))(2-(\bs_j,\bs_j))$ is less than $1/4$,
	 	then $|(\bv_i,\bv_j)|=0$ or $2$.
	\end{enumerate}
\end{lem}
\begin{proof}
	(\ref{lem:v:1})
	It suffices to show that
	$\bv_1,\ldots,\bv_m$ are pairwise distinct.
	To prove by contradiction, we suppose that $i \neq j$ and $\bv_i = \bv_j$.
	Let $A$ be the Gram matrix with respect to $\bv_i$ and $\bv_j$,
	and $G$ the Gram matrix with respect to $\bs_i$ and $\bs_j$.
	Then, $2I-G$ is positive semi-definite
	since $\bs_i$ and $\bs_j$ are elements of a eutactic star of scale $2$.
	In addition, $A-G$ is also postive semi-definite.
	By applying Lemma~\ref{lem:gram}~(\ref{c2}) with $\delta:=2$,
	we have
	$
		(2-(\bs_i,\bs_i))(2-(\bs_j,\bs_j)) \geq (2/2)^2 = 1.
	$
	Since the norms of $\bs_i$ and $\bs_j$ are greater than $1$,
	the left hand side is smaller than $1$.
	This is a contradiction.
	Thus, $\bv_i \neq \bv_j$ holds,
	and the desired conclusion follows.
	(\ref{lem:v:2})
	Let $i$ and $j$ be integers in $\{1,\ldots, m \}$.
	To prove the contrapositive,
	we suppose that $(\bv_i,\bv_j)= \pm 1$ holds.
	Let $A$ be the Gram matrix with respect to $\bv_i$ and $\bv_j$,
	and $G$ that of $\bs_i$ and $\bs_j$.
	Then $2I-G$ and $A-G$ are positive semi-definite.
	By applying Lemma~\ref{lem:gram} with $\delta:= \pm1$,
	we have
	$
		(2-(\bs_i,\bs_i))(2-(\bs_j,\bs_j)) \geq (\pm 1/2)^2 = 1/4.
	$
	This is the desired result.
\end{proof}
	
\section{The lattice $\A_{15}^+$}		\label{sec:a}
The lattice $\A_{15}^+$ is given in Definition~\ref{dfn:A15p}.
Let $\sR$ be the set of elements in $\A_{15}^{+}$ of norm $2$.
Since the minimal norm in the nonzero cosets of $\A_{15}$ in $\A_{15}^{+}$ is $3$,
$\sR$ is also the set of elements in $\A_{15}$ of norm $2$.		
For a set $X$, we let $\Sym(X)$ denote the symmetric group on $X$.
For a positive integer $n$, let $S_{n}$ denote the symmetric group $\Sym(\{1,\ldots,n\})$ of degree $n$.
The symmetric group $S_{16}$ of degree $16$ acts on $\A_{15}^{+}$
such that,
for $\bx \in \A_{15}^{+}$, $\sigma \in S_n$ and $i \in \{1,\ldots,16\}$,
the $\sigma(i)$-th entry of $\sigma(\bx)$ is defined by the $i$-th entry of $\bx$.
In fact, $\Aut(\A_{15}^+)=\langle S_{16},-1 \rangle$ holds.

In this section we discuss properties of the lattice $\A_{15}^+$ and its non $2$-integrable sublattices.
As claimed in Lemma~\ref{lem:12}, every non $2$-integrable lattice of rank $12$ is contained in $\A_{15}^+$.
Lemma~\ref{lem:thm} is the first statement of Theorem~\ref{thm:main},
which asserts our newly found lattices are $\langle \ba,\bb,\bc \rangle^{\perp}$ and $\langle \ba,\bb,\bc' \rangle^{\perp}$.
In Lemma~\ref{lem:M12}, we explain properties of these lattices.

\begin{lem}	\label{lem:thm}
	There are precisely two sublattices in $\A_{15}^+$ up to $\Aut(\A_{15}^+)$ with Gram matrix~(\ref{thm:main:1}).
	Furthermore, they are $\langle \ba,\bb,\bc \rangle$ and $\langle \ba,\bb,\bc' \rangle$ in $\A_{15}^+$.
\end{lem}	
\begin{proof}
	Let $T$ be the set of elements of norm $3$ in $\A_{15}^+$.		
	For pairwise distinct integers $i_1,i_2,i_3$ and $i_4 \in \{1,\ldots,16\}$,
	we let
	\begin{align*}
		\bt_{i_1,i_2,i_3,i_4} = \bt_{\{i_1,i_2,i_3,i_4\}}
		:=
		(1/4) \e
		- \be_{i_1} -\be_{i_2}-\be_{i_3}-\be_{i_4} \in \A_{15}^+,
	\end{align*}
	where $\e$ denotes the all one vector in $\Z^{16}$.
	For example the vector $[4]$ defined in Definition~\ref{dfn:A15p} is $\bt_{13,14,15,16}$.		
	First,
	we show that
	\begin{align}	\label{lem:thm:1}
		T = \{ \pm \bt_I \mid I \subseteq \{1,\ldots,16\} \text{ and } |I| =4 \}.
	\end{align}		
	As the representatives of cosets of $\A_{15}$ in $\A_{15}^+$ are $0$, $\pm [4]$ and $2[4]$,
	and the norm of every element in $\A_{15}$ and $2[4]+\A_{15}$ is even, every element in $T$ must belong to $\pm[4] +\A_{15}$. 
	Let $\by=(y_1,\ldots,y_{16}) \in \A_{15}$, and suppose $[4] + \by \in T$.
	Then we obtain two conditions
	\begin{align*}
		\sum_{i=1}^{12} (4y_i+1) + \sum_{j=13}^{16}(4y_j-3) = 0
		\text{\quad and \quad}
		\sum_{i=1}^{12} (4y_i+1)^2 + \sum_{j=13}^{16}(4y_j-3)^2 = 3 \cdot 4^2 = 48.
	\end{align*}
	By the second condition, the odd integers $4y_i+1$ and $4y_j-3$ clearly belong to $\{-3,1,5\}$ for all $i$ and $j$.
	This together with the first condition implies that they belong to $\{-3,1\}$.
	Thus there exists $I \subseteq \{1,\ldots,16\}$ with $|I|=4$
	such that $[4] + \by = \bt_I$.
	This implies
	\begin{align*}
		\left( [4]+\A_{15}^+ \right) \cap T \subseteq \{ \bt_I \mid I \subseteq \{1,\ldots,16\} \text{ and } |I| =4 \}.
	\end{align*}
    As $-[4]+\A_{15}^+=-([4]+\A_{15}^+)$, it comes with 
    $$\left(-[4]+\A_{15}^+ \right) \cap T \subseteq \{-\bt_I \mid I \subseteq \{1,\ldots,16\} \text{ and } |I| =4 \}$$
	and the equality~(\ref{lem:thm:1}) holds.
	Next,
	we classify three elements $\bx,\by$ and $\bz$ of norm $3$ in $\A_{15}^+$ up to $\Aut(\A_{15}^+)$ such that the Gram matrix with respect to them is (\ref{thm:main:1}).
	We can let $\bx := \bt_{1,2,3,4}=\ba$.
	For subsets $I$ and $J$ of cardinality $4$ in $\{1,\ldots,16\}$,
	we have
	$
		(\bt_I,\bt_J) = -1 + |I \cap J|.
	$
	Hence, we let $\by:=\bt_{1,2,3,5}=\bb$ to satisfy $(\bx,\by)=2$.
	Similarly, to satisfy $(\bx,\bz)=(\by,\bz)=0$,
	we let $\bz:=\bt_{1,6,7,8}=\bc$ or $\bz := \bt_{4,5,6,7}=\bc'$.
	The desired conclusion holds.
\end{proof}

\begin{lem}	\label{lem:M12}
	Let $L$ be a unimodular lattice.
	Let $M$ be a sublattice of $L$ with Gram matrix~\eqref{thm:main:1}.
	Let $\pr$ denote the orthogonal projection from $L$ to $M^{\perp}\otimes\Q $.
	Then, the following holds:
	\begin{enumerate}
		\item	$M$ is a primitive sublattice in $L$.	\label{b1}
		\item	$\det(M) = \det(M^{\perp})= 15$.	\label{b2}
		\item	The minimal norms of the representatives for the nonzero cosets of $M$ in $M^*$
			are $1/3$, $2/5$, $3/5$, $11/15$ and $14/15$.	\label{b3}
		\item	\label{b4}
			Suppose that the minimal norm of $L$ is at least $2$.
			Then,
			the minimum norm of $(M^\perp)^*$ is at least $16/15$.
			Furthermore,
			for every nonzero element $\bw \in (M^{\perp})^*$ of norm at most $2$,	
			there exists $\bw' \in L$ of norm $2$
			such that $\bw = \pr(\bw')$.
	\end{enumerate}
\end{lem}
\begin{proof}
	\eqref{b1} and \eqref{b2} follow from Lemma~\ref{lem:pri3} and Lemma~\ref{lem:pri4}, respectively.
	\eqref{b3} Let $\bv_1,\bv_2,\bv_3$ be the elements of $M$ with \eqref{thm:main:1} as their Gram matrix. Then $M=\langle\bv_1,\bv_2 \rangle\perp \langle\bv_3\rangle$ and $M^*=\langle\bv_1,\bv_2\rangle^*\perp\langle\bv_3\rangle^*$. The representatives for the nonzero cosets of $M$ in $M^*$
	are \[
	\pm(\bv_1+\bv_2)/5, \pm(3\bv_1-2\bv_2)/5, \pm\bv_3/3, \pm(\bv_1+\bv_2)/5\pm\bv_3/3, \pm(3\bv_1-2\bv_2)/5\pm\bv_3/3,
	\]
	and their norms are $2/5,3/5,1/3,11/15$ and $14/15$, respectively. Next we show \eqref{b4}. Take any element $\bw \in (M^{\perp})^*$ of norm at most $2$.
	Let $\bu_0 , \bu_1 , \ldots, \bu_{14}$ be the representatives for the nonzero cosets of $M$ in $M^*$ which are with minimal norm.
	By applying Lemma~\ref{lem:pri2} to $L$ and $M$,
	we have
	\begin{align*}
		(M^{\perp})^* \perp M^* = L + M^* = \bigsqcup_{i=0}^{14} \left( \bu_i + L \right) .
	\end{align*}
	Hence,
	$
		\bw  \in \bu_i + L
	$
	for some integer $i$.
	Then,
	we have $\bw-\bu_i \in L$ and
	\begin{align*}
			0 < |\bw-\bu_i|^2 = |\bw|^2+|\bu_i|^2 \leq 2 + 14/15 < 3.
	\end{align*}
	Since the minimum norm of $L$ is at least $2$,
	$|\bw-\bu_i|^2 = 2$ follows.
	Letting $\bw':=\bw-\bu_i$, we have $\bw=\pr(\bw')$ and the norm of $\bw'$ is $2$.		
	Moreover, we have
	$
		|\bw|^2 = 2-|\bu_i|^2 \geq 2-14/15 = 16/15.
	$
	Hence, the minimum norm of $(M^\perp)^*$ is at least $16/15$.
\end{proof}		

For an element $\bx=(x_1,\ldots,x_n)$ in $\R^n$,
the \emph{support} of $\bx$, denoted by $\supp	\bx$, is the set of $i \in \{1,\ldots, n\}$ with $x_i \neq 0$.
We fix a partition $\pi = \{X_i\}_{i=1}^q$ of $\{1,\ldots,16\}$.
Then a group $A_\pi$ is defined as the subgroup of $\Aut(\A_{15}^+)$ generated by $-1$ and $\Sym(X_i)$ for all $i \in \{1,\ldots,q\}$.
For an element $\bx \in \sR = \{ \by \in \A_{15}^+ \mid (\by,\by) = 2\}$, the \emph{type} of $\bx$ for $\pi$ is defined by
\begin{align*}
	\type_\pi(\bx) := \{ i \in \{1,\ldots,q\} \mid \supp \bx \cap X_i \neq \emptyset \}.
\end{align*}
In addition, for any $\bx$ and $\by$ in $\sR$, we let
\begin{align*}
	\type_\pi(\bx,\by) := \{ i \in \{1,\ldots,q\} \mid \supp \bx \cap \supp \by \cap X_i \neq \emptyset \}.
\end{align*}
Then we have the following lemma.

\begin{lem}	\label{lem:type_inner}
	Let $\pi$ be a partition of $\{1,\ldots,16\}$.
	Let $N$ be a sublattice of $\A_{15}^{+}$ invariant under $A_\pi$.
	Let $\pr$ denote the orthogonal projection from $\A_{15}^{+}$ to $N\otimes\Q$.
	For two elements $\bu$ and $\bv$ in $\sR$,
	the value $|(\pr(\bu),\pr(\bv))|$ depends only on
	\begin{enumerate}
		\item		\label{lem:type_inner:1}
			$\type_\pi(\bu)$ and $\type_\pi(\bv)$ if $|(\bu,\bv)|=2$, and
		\item		\label{lem:type_inner:2}
			$\type_\pi(\bu)$, $\type_\pi(\bv)$ and $\type_\pi(\bu,\bv)$ if $|(\bu,\bv)| \leq 1$.
	\end{enumerate}
\end{lem}
\begin{proof}
	Write $\bu = \be_j - \be_k$ and $\bv = \be_l -\be_m$.
	Take arbitrary two elements $\bu' = \be_{j'}-\be_{k'}$ and $\bv' =  \be_{l'}-\be_{m'}$ in $\sR$
	which satisfy $|(\bu,\bv)|=|(\bu',\bv')|$,
	$\type_\pi(\bu) = \type_\pi(\bu')$ and $\type_\pi(\bv) = \type_\pi(\bv')$.
	In addition, we suppose that $\type_\pi(\bu,\bv) = \type_\pi(\bu',\bv')$ if $|(\bu,\bv)| \leq 1$.
	We show that there exists $\sigma \in A_{\pi}$ so that
	\begin{align}	\label{lem:type_inner:3}
		\bu' = \sigma(\bu) \text{ and } \bv' = \pm \sigma(\bv).
	\end{align}
	Consider the case that \eqref{lem:type_inner:1} is satisfied.
	Without loss of generality we may assume that $\{j,j'\}\subseteq C$ (resp. $\{k,k'\}\subseteq D$) for some $C \in \pi$ (resp. $D \in \pi$), as $\type_\pi(\bu) = \type_\pi(\bu')$.
	Let $\sigma = (j,j')(k,k')$. Then $\sigma \in A_\pi$ and $\bu = \pm \sigma(\bu')$. Since $|(\bu,\bv)|=|(\bu',\bv')|=2$, it is immediately to see $\supp \bu=\supp \bv$ and $\supp \bu'=\supp \bv'$, and thus $\bv = \pm \sigma(\bv')$.
	In the case that \eqref{lem:type_inner:2} is satisfied, the existence of $\sigma$ is also easily verified.
	
	We consider the automorphism group $\Aut(\A_{15}^+)$ as acting on $\A_{15}^+\otimes \R $.
	For $\sigma \in A_{\pi}$ satisfying equality~\eqref{lem:type_inner:3}, we have
	\begin{align*}
		|(\pr(\bu),\pr(\bv))| = |(\sigma(\pr(\bu)),\sigma(\pr(\bv)))| = |(\pr \circ \sigma(\bu),\pr \circ \sigma(\bv))| = |(\pr(\bu'),\pr(\bv'))|.
	\end{align*}
	This is the desired conclusion.
\end{proof}

The following lemma gives a necessary condition for a lattice to be $2$-integrable.
In Section~\ref{sec:proof},
it turns out that we can apply Lemma~\ref{lem:m}
to our lattices $\langle \ba,\bb,\bc \rangle^\perp$ and $\langle \ba,\bb,\bc' \rangle^\perp$.
As a result, we conclude that they are non $2$-integrable.	

\begin{lem}	\label{lem:m}
	Let $X \subseteq \{1,\ldots,16\}$ with $|X| \geq 3$.
	Let $N$ be a sublattice of $\A_{15}^+$, and $\pr$ the orthogonal projection from $\A_{15}^+$ to $N\otimes\Q$.
	Suppose that the following conditions are satisfied:
	\begin{enumerate}
		\item	\label{lem:m:1}
			The minimum norm of $N^*$ is greater than $1$.
		\item	\label{lem:m:2}
			$\pr(\sR)$ contains the nonzero elements in $N^*$ of norm at most $2$.
		\item	\label{lem:m:3}
			$N$ contains the elements in $\sR$ whose support is contained in $X$.
	\end{enumerate}
	If $N$ is $2$-integrable,
	then
	there exist $\bu$ and $\bv$ in $\sR$ with $\bu \neq \pm \bv$ such that the following hold.
	\begin{enumerate}
		\setcounter{enumi}{3}
		\item	\label{lem:m:4}
			$\supp \bu \cap \supp \bv \cap X \neq \emptyset$, 
			and $(2-(\pr(\bu),\pr(\bu))(2-(\pr(\bv),\pr(\bv)) \geq 1/4$.
		\item	\label{lem:m:5}
			$2I-G$ is positive semi-definite, where $G$ is the Gram matrix with respect to $\pr(\bu)$ and $\pr(\bv)$.
	\end{enumerate}
\end{lem}
\begin{proof}
	Since $N$ is $2$-integrable,
	by applying Theorem~\ref{thm:iff_eut} to $N$,
	there exist nonzero elements $\bs_1,\ldots, \bs_m$ in $N^*$ which form a eutactic star of scale $2$.
	The norms of $\bs_1,\ldots,\bs_m$ are greater than $1$ by the condition~(\ref{lem:m:1}).
	By the condition~(\ref{lem:m:2}),
	we can find $\bv_1,\ldots,\bv_m$ in $\sR$ such that $\bs_i = \pr(\bv_i)$ for every $i$.
	Lemma~\ref{lem:v}~\eqref{lem:v:2} together with the condition~(\ref{lem:m:1}) implies that
	$\bv_1,-\bv_1,\ldots,\bv_m,-\bv_m$ are pairwise distinct.
	
	We show that
	there exist two distinct $i$ and $j$ in $\{1,\ldots ,m\}$
	such that $\bu:=\bv_i$ and $\bv:=\bv_j$ satisfy the conditions~(\ref{lem:m:4}) and (\ref{lem:m:5})
	as follows:
	First,
	since $\bs_1,\ldots,\bs_m$ form a eutactic star of scale $2$, the condition~(\ref{lem:m:5}) is satisfied for all $\bu:=\bv_i$ and $\bv:=\bv_j$ with $i \neq j$.	
	
	Next, we suppose that
	\begin{align*}
		\supp \bv_i \cap \supp \bv_j \cap X = \emptyset
	\end{align*}
	for all $i$ and $j$ with $i \neq j$.
	Since $|X| \geq 3$,
	there exist two distinct $i_1$ and $i_2$ such that
	\begin{align*}
		\{ i_1,i_2\} \in \binom{X}{2} \setminus \{ \supp \bv_i \mid i = 1,\ldots,m \}.
	\end{align*}
	Let $\bw := \be_{i_1}-\be_{i_2}$.
	Then $\bw \in N$ holds by the condition~(\ref{lem:m:3}).
	For each $i$,
	since
	\begin{align*}
		(\bw,\bw)(\bs_i,\bs_i)>2\text{ and } (\bw,\bs_i)^2 = (\bw,\bv_i)^2 \in \{0,1\},
	\end{align*}
	we have $(\bw,\bw)(\bs_i,\bs_i) > (\bw,\bs_i)^2$ and thus $\bs_i \in N^* \setminus \Q \bw$.
    By applying Lemma~\ref{lem:hoge} with $s := 2$,
	we have
	\begin{align*}
		4 = 2 \cdot 2
		&\leq |\{i \in \{1,\ldots,m\} \mid (\bw,\bs_i) \neq 0\}| \cdot \lceil \sqrt{2 \cdot 2}-1 \rceil^2	\\
		&= |\{i \in \{1,\ldots,m\} \mid (\bw,\bv_i) \neq 0\}| \cdot 1\\
		&\leq |\{i \in \{1,\ldots,m\} \mid \supp \bw \cap \supp \bv_i \neq \emptyset\}|	\\
		&\leq \sum_{l=1,2} |\{i \in \{1,\ldots,m\} \mid i_l \in \supp \bv_i\} |	\\
		& \leq 1+1 = 2
	\end{align*}
	a contradiction.
	Hence there exist $i$ and $j$ with $i \neq j$ such that $\supp \bv_i \cap \supp \bv_j \cap X \neq \emptyset$.
	Thus we let $\bu := \bv_i$ and $\bv := \bv_j$.
	Then since $|(\bu,\bv)|=1$,
	Lemma~\ref{lem:v}~\eqref{lem:v:2} implies that $(2-(\pr(\bu),\pr(\bu))(2-(\pr(\bv),\pr(\bv)) \geq 1/4$.
	Therefore \eqref{lem:m:4} is obtained.
\end{proof}
	
\section{Minimal non $s$-integrable lattices}   \label{sec:minimality}
%		For each positive integer $s$, there are infinitely many non $s$-integrable lattices.
%		Hence the minimality of non $s$-integrable lattices was studied.
		In this section, we prove Proposition~\ref{prop:P} which will be used to show the minimality of the sublattices $\langle \ba,\bb,\bc \rangle^\perp$ and $\langle \ba,\bb,\bc' \rangle^\perp$ in $\A_{15}^+$ in Proposition~\ref{prop:main_min}.
		It turns out that these non $2$-integrable lattices are not essentially obtained from Conway and Sloane's non $2$-integrable lattices in Theorem~\ref{thm:disc7}.
				
		Plesken~\cite{P} studied minimal non $1$-integrable lattices and additively indecomposable ones defined in the following.
		Note that he calls the bilinear form corresponding to a minimal non $1$-integrable lattice a \emph{block form}.
		We state his claims in terms of lattice theory.
		
		\begin{dfn}
			A lattice $L$ is said to be \emph{additively decomposable} if there are two lattices $M$ and $N$ such that $L$ is isometric to a sublattice of $M \perp N$ which is contained in neither $N$ nor $M$.
			Otherwise it is said to be \emph{additively indecomposable}.
		\end{dfn}
		
		\begin{lem}[{\cite[(II.5)~COROLLARY]{P}}]	\label{lem:P}
			A lattice $L$ is minimal non $1$-integrable if and only if the minimum norm of $L^*$ is greater than $1$.
		\end{lem}			
			
		Moreover, Plesken gave a sufficient condition for a lattice to be additively indecomposable (see \cite[(III.1)~PROPOSITION]{P}).
		With a slight change in his argument, the following lemma is derived.
		Note that a lattice is said to be \emph{irreducible}
		if it is not the orthogonal sum of two nonzero lattices.
		\begin{prop}		\label{prop:P}
			Let $L$ be a minimal non $1$-integrable lattice.
			Suppose that there is an irreducible sublattice of rank at least $\rank L- 5$ which is generated by elements of norm at most $3$.
			Then $L$ is additively indecomposable.
		\end{prop}		
		\begin{proof}
			Suppose that there exists such a sublattice $L'$.
			By way of contradiction,
			we suppose that $L$ is additively decomposable.
			Thus there are two lattices $M$ and $N$ such that $L \subseteq M \perp N$, $L \not\subseteq M$ and $L \not\subseteq N$.
			Let $\pr_M$ and $\pr_N$ denote the orthogonal projections to $M$ and $N$, respectively.
			
			First we show that either $L' \subseteq M$ or $L' \subseteq N$.
			It suffices to show that there is no element $\bu \in L'$ of norm at most $3$ such that $\pr_M(\bu) \neq 0$ and $\pr_N(\bu) \neq 0$.
			Suppose that there exists such an element $\bu$.
			Without loss of generality, we may assume the norm of $\pr_N(\bu)$ is equal to $1$.
			Then $N = N' \perp  \langle \pr_N(\bu) \rangle$ for a sublattice $N'$ of $N$.
			Therefore
			\[
				L \subseteq M \perp N \subseteq (M \perp N') \perp \langle \pr_N(\bu) \rangle \simeq (M \perp N') \perp \Z
				\text{ and }
				L \not\subseteq M \perp N'.
			\]
			This means that $L$ is non $1$-minimal, which gives a contradiction.
			
			Now we may assume that $L' \subseteq M$.
			Set $P := (L'\otimes \Q)\cap L$.
			Then $P$ is a primitive sublattice of $L$ and $L = P \oplus Q$ for some sublattice $Q$ of $L$.		
			Since $P \subseteq M\otimes\Q$, we have
			\[
				L \subseteq M \perp \pr_N(L) = M \perp \pr_N( Q ).
			\]
			As
			\[
				\rank \pr_N(Q) \leq \rank Q = \rank L - \rank P = \rank L - \rank L' \leq 5,
			\]
			this together with Theorem~\ref{embedding to odd} implies that $\pr_N(Q)$ is a sublattice of an odd unimodular lattice of rank at most $8$.  It is well-known that every odd unimodular lattice of rank $k\leq 8$ is isometric to standard lattice $\Z^k$ (see~\cite[Table~16.7]{CS1999}). Thus $\pr_N(Q)\subseteq \Z^8$. Furthermore, $L \subseteq M \perp \Z^8$ and $L \not\subseteq M$.
			This means that $L$ is non $1$ minimal, which leads a contradiction.
			Thus, the desired conclusion holds.
		\end{proof}

\section{Proof of Theorem~\ref{thm:main}}	\label{sec:proof}
In this section, we will prove Theorem~\ref{thm:main}, which follows from Lemma~\ref{lem:thm} and Proposition~\ref{prop:main_min} immediately.

\begin{prop}		\label{prop:N1}
	The sublattice in $\A_{15}^+$ orthogonal to $\langle \ba,\bb,\bc \rangle$ is non $2$-integrable.
\end{prop}
\begin{proof}
	Let $N$ be the sublattice in $\A_{15}^+$ orthogonal to $\langle \ba,\bb,\bc \rangle$.
	Let $\pr$ denote the orthogonal projection from $\A_{15}^+$ to $N\otimes\Q$.
	Lemma~\ref{lem:M12}~(\ref{lem:m:4}) asserts that
	$N$ satisfies the two conditions~(\ref{lem:m:1}) and (\ref{lem:m:2}) in Lemma~\ref{lem:m}.
	Let $X:=\{9,10,\ldots,16\}$.
	Then, the condition~(\ref{lem:m:3}) in Lemma~\ref{lem:m} is satisfied.
	Hence, we apply Lemma~\ref{lem:m} to $N$ and $X$ to prove that $N$ is non $2$-integrable.
	Indeed, it suffices to prove that
	$2I-G$ is non positive semi-definite,
	where $G$ is the Gram matrix with respect to $\pr(\bu)$ and $\pr(\bv)$,
	for any $\bu$ and $\bv$ in $\sR$ with $\bu \neq \pm \bv$, $\supp \bu \cap \supp \bv \cap X \neq \emptyset$
	and 
	\begin{align}	\label{prop:N1:1}
		(2-(\rho(\bu),\rho(\bu)))(2-(\rho(\bv),\rho(\bv))) \geq 1/4.
	\end{align}
	Fix such elements $\bu$ and $\bv$.
	Let $\pi$ be the partition of $\{1,\ldots,16\}$ consisting of
	\[	
	X_1 := \{1\},~X_2 := \{2,3\},~X_3 := \{4,5\},~X_4 := \{6,7,8\} \text{ and }X_5 := X.
		\]
	Then, $\langle \ba,\bb,\bc \rangle$ is invariant under $A_{\pi}$,
	and hence so is $N$.
	%Moreover, $\type_{\pi}(\bu,\bv) = \{5\}$ holds since $|\supp\bu|=|\supp\bv|=2$, $\supp \bu \cap \supp \bv \cap X=\supp \bu \cap \supp \bv \cap X_5\neq \emptyset$ and $\bu \neq \pm \bv$.
	Lemma~\ref{lem:type_inner}~\eqref{lem:type_inner:1} implies that the norm of $\rho(\bx)$ depends only on the $\type_{\pi}(\bx)$ for every $\bx \in \sR$.
	Note that
	\begin{align*}
		\pr(\be_1-\be_{16})
		&= \frac{1}{60}( 27,-13,-13,-1,-1,-9,-9,-9,11,11,11,11,11,11,11,-49 ),	\\
		\pr(\be_2-\be_{16}) 				
		&= \frac{1}{10}(-3,7,-3,-1,-1,1,1,1,1,1,1,1,1,1,1,-9),	\\
		\pr(\be_4-\be_{16})
		&=
		\pr(\be_5-\be_{16})
		= \frac{1}{20}(-3,-3,-3,9,9,1,1,1,1,1,1,1,1,1,1,-19), \\
		\pr(\be_6-\be_{16})
		&= \frac{1}{12}(-3,1,1,1,1,9,-3,-3,1,1,1,1,1,1,1,-11) \text{ and }\\
		\pr(\be_9-\be_{16})
		&= \be_9-\be_{16},
	\end{align*}
	and their norms are $19/15$, $8/5$,$7/5$, $5/3$ and $2$, respectively.
	By~\eqref{prop:N1:1}, without loss of generality we may assume that the pair of the type for $\pi$ of $\bu$ and $\bv$ is one of
	$(\{1,5\},\{1,5\})$, $(\{1,5\},\{2,5\})$, $(\{1,5\},\{3,5\})$ and $(\{3,5\},\{3,5\})$.
	In the case of $\type_{\pi}(\bu) = \type_{\pi} (\bv) = \{1,5\}$, it is easy to see $\bu=\pm\bv$ and this is impossible. By applying Lemma~\ref{lem:type_inner} to $\bu$ and $\bv$,
	the absolute value of the inner product $|(\pr(\bu),\pr(\bv))|$ depends only on the pair of type for $\pi$ of $\bu$ and $\bv$
	since $\type_{\pi}(\bu,\bv)=\{5\}$.
	Thus,
	the Gram matrices $G$ with respect to $\pr(\bu)$ and $\pr(\bv)$ is one of
	\begin{align*}
		\BB{19/15}{\pm 3/5}{\pm 3/5}{8/5},
		\BB{19/15}{\pm 4/5}{\pm 4/5}{7/5} \text{ and }
		\BB{7/5}{\pm 7/5}{\pm 7/5}{7/5}.
	\end{align*}			
	Then $2I-G$ is non positive semi-definite.
	The desired conclusion follows.
\end{proof}

\begin{prop}		\label{prop:N2}
	The sublattice in $\A_{15}^+$ orthogonal to $\langle \ba,\bb,\bc' \rangle$ is non $2$-integrable.
\end{prop}
\begin{proof}
	Let $N'$ be the sublattice in $\A_{15}^+$ orthogonal to $\langle \ba,\bb,\bc' \rangle$.
	Let $\pr$ denote the orthogonal projection from $\A_{15}^+$ to $N'\otimes\Q$.
	Let $X:=\{8,9,\ldots,16\}$.	
	As in the case of the proof of Proposition~\ref{prop:N1},
	we apply Lemma~\ref{lem:m} to $N$ and $X$.
	To prove that $N$ is non $2$-integrable,
	 it suffices to prove that
	$2I-G$ is non positive semi-definite,
	where $G$ is the Gram matrix with respect to $\pr(\bu)$ and $\pr(\bv)$,
	for any $\bu$ and $\bv$ in $\sR$ with $\bu \neq \pm \bv$, $\supp \bu \cap \supp \bv \cap X \neq \emptyset$,
	and 
	\begin{align}	\label{prop:N2:1}
		(2-(\rho(\bu),\rho(\bu)))(2-(\rho(\bv),\rho(\bv))) \geq 1/4.
	\end{align}
	Fix such elements $\bu$ and $\bv$.
	Let $\pi$ be the partition of $\{1,\ldots,16\}$ consisting of
	 	\[	
	 X_1 := \{1,2,3\},~X_2 := \{4,5\},~X_3 := \{6,7\} \text{ and }X_4:= X.
	 \]
	Note that
	\begin{align*}
		\pr(\be_1-\be_{16})
		&= \frac{1}{10}(7,-3,-3,-1,-1,1,1,1,1,1,1,1,1,1,1,-9 ),	\\
		\pr(\be_4-\be_{16}) 				
		&= \pr(\be_5-\be_{16})				
		= \frac{1}{15}(-1,-1,-1,3,3,-3,-3,2,2,2,2,2,2,2,2,-13), 	\\
		\pr(\be_6-\be_{16})
		&= \frac{1}{12}(1,1,1,-3,-3,9,-3,1,1,1,1,1,1,1,1,-11) \text{ and } \\
		\pr(\be_8-\be_{16})
		&=	\be_8-\be_{16},
	\end{align*}
	and their norms are $8/5$, $16/15$, $5/3$ and $2$, respectively.
	By~\eqref{prop:N2:1}, without loss of generality we may assume that the pair of the type for $\pi$ of $\bu$ and $\bv$ is one of
	$(\{1,4\},\{2,4\})$, $(\{2,4\},\{2,4\})$ and $(\{2,4\},\{3,4\})$.
	By applying Lemma~\ref{lem:type_inner} to $\bu$ and $\bv$,
	the absolute value of the inner product $|(\pr(\bu),\pr(\bv))|$ depends only on the pair of type for $\pi$ of $\bu$ and $\bv$.
	Thus,
	the Gram matrices $G$ with respect to $\pr(\bu)$ and $\pr(\bv)$ is one of
	\begin{align*}
		\BB{8/15}{\pm 4/5}{\pm 4/5}{16/15},
		\BB{16/15}{\pm 16/5}{\pm 16/5}{16/5} \text{ and }
		\BB{16/15}{\pm 2/3}{\pm 2/3}{5/3}.
	\end{align*}			
	Then $2I-G$ is not positive semi-definite.
	The desired conclusion follows.
\end{proof}

\begin{prop}	\label{prop:main_min}
	The sublattices $\langle \ba,\bb,\bc \rangle^\perp$ and $\langle \ba,\bb,\bc' \rangle^\perp$ in $\A_{15}^+$ are non-isometric and minimal non $2$-integrable lattices.
\end{prop}		
\begin{proof}
	Set $N := \langle \ba,\bb,\bc \rangle^\perp$ and $N' := \langle \ba,\bb,\bc' \rangle^\perp$.
	Then
	\begin{align}		\label{eq:N2}	
		N \cap \sR &= \bigsqcup_{Y \in \tau} \{ \bx \in \sR \mid \supp \bx \subseteq Y \} 
		= \{ \be_i-\be_j \mid i \neq j \text{ and } i,j \in Y \text{ for some } Y \in \tau  \},
	\end{align}
	where $\tau := \{\{1\},\{2,3\},\{4\},\{5\},\{6,7,8\},\{9,\ldots,16\}\}$.
	Since	the minimum norm of $N$ is $2$, the kissing number of $N$ is $|N \cap \sR|$.
	Hence the kissing number of $N$ is $2 \cdot 1 + 3 \cdot 2 + 8 \cdot 7 = 64$.
	Similarly,
	\begin{align}	\label{eq:N'2}
		N' \cap \sR = \bigsqcup_{Y \in \tau'} \{ \bx \in \sR \mid \supp x \subseteq Y \},
	\end{align}
	where $\tau' := \{\{1,2,3\},\{4\},\{5\},\{6,7\},\{8,\ldots,16\}\}$,
	and the kissing number of $N'$ is $3 \cdot 2 + 2 \cdot 1 + 9 \cdot 8 = 80$.	
	Hence $N$ and $N'$ are non-isometric.
	
	Since Proposition~\ref{prop:N1} and Proposition~\ref{prop:N2} claim that $N$ and $N'$ are non $2$-integrable,
	it suffices to prove the minimality of them.
	Lemma~\ref{lem:M12} implies that the minimum norms of $N^*$ and $(N')^*$ are at least $16/15$.
	Thus, by Lemma~\ref{lem:P}, they are minimal non $1$-integrable lattices.
	By applying Proposition~\ref{prop:P} with $L := N$ and $L := N'$,
	we prove that $N$ and $N'$ are additively indecomposable.
	In particular, they are minimal non $2$-integrable.
	Namely, it suffices to show that each of $N$ and $N'$ contains an irreducible sublattice of rank at least $7 = 12-5$ generated by elements of norm at most $3$.
	By~\eqref{eq:N2} and \eqref{eq:N'2},
	both $N$ and $N'$ contains
	\[
		\langle \be_9-\be_{10}, \ldots, \be_{15}-\be_{16} \rangle
	\]

	as a sublattice of rank $7$.
	Therefore the desired conclusion follows.
\end{proof}			
		
Plesken~\cite{P} has proved that  $\langle \ba,\bb,\bc' \rangle^\perp$ is additively indecomposable (see \cite[(III.3)~EXAMPLE]{P}),
where $\langle \ba,\bb,\bc' \rangle^\perp$ is written by $1^8,2^3;6$.
However, it is not obvious that our definition of $\langle \ba,\bb,\bc' \rangle^\perp$ and his are same.	
	
\begin{rem}		\label{rem:new strategy}
	Conway and Sloane~\cite{CS1989} proved Theorem~\ref{thm:disc7}
	by using Lemma~\ref{lem:eut} and choosing test vectors $\bw$ well.
	We can also prove it by the same strategy as the proof of Theorem~\ref{thm:main} as follows:
	By the similar argument as the proof of Lemma~\ref{lem:thm},
	there are precisely two lattices with Gram matrix~(\ref{thm:disc7:1}) up to $\Aut(\A_{15}^+)$ in $\A_{15}^+$,
	and they are given by $\langle \ba,\bb,\bc'' \rangle$ and $\langle \ba,\bb,\bc''' \rangle$,
	where
	\begin{align*}
		\bc'' &:= \frac{1}{4}(-3,-3, -1,-3,-3,	 1,1,1,1,1,1,1,1,1,1)	\text{ and}		\\
		\bc''' &:= \frac{1}{4}(-3,-3,-3,1,1,-3,		1,1,1,1,1,1,1,1,1,1).
	\end{align*}				
	Let $N'' := \langle \ba,\bb,\bc'' \rangle^\perp$, $N''' := \langle \ba,\bb,\bc''' \rangle^\perp$, $Y:=\{6,\ldots,16\}$ and $Z:=\{7,\ldots,16\}$.
	By considering the analogue of Lemma~\ref{lem:M12} where Gram matrix~(\ref{thm:main:1}) is replaced by the Gram matrix~(\ref{thm:disc7:1}),
	it follows that $N''$ and $N'''$ satisfy the conditions~(\ref{lem:m:1}) and (\ref{lem:m:2}) in Lemma~\ref{lem:m}.
	In addition,
	the condition~(\ref{lem:m:3}) in Lemma~\ref{lem:m} is satisfied
	when we let $(N,X):=(N'',Y)$ and $(N''',Z)$.
	Hence,
	we can apply Lemma~\ref{lem:m} with $(N,X):=(N'',Y)$ and $(N''',Z)$.
	Therefore	we can prove the non $2$-integrability by the similar argument as the proof of Proposition~\ref{prop:N1} and Proposition~\ref{prop:N2} without using computer.
\end{rem}
			
\section*{Acknowledgements}
We would like to express our sincere gratitude to Professor Munemasa for his helpful comments.


\begin{thebibliography}{99}
	\bibitem{cassels}
	J.~W.~S.~Cassels,
	\emph{Rational Quadratic Forms},
	Academic Press, London; New York (1978).
	
	\bibitem{CS1988}
	J.~H.~Conway, N.~J.~A.~Sloane,
	Low-Dimensional Lattices I: Quadratic Forms of Small Determinant,
	\emph{Proc.~R.~Soc.~London, Ser.~A}
	\textbf{418} (1988) 17--41.
	\bibitem{CS1989}
	J.~H.~Conway, N.~J.~A.~Sloane,
	Low-dimensional lattices V: Integral coordinates for integral lattices.
	\emph{Proc.~Roy.~Soc.~ London Ser.~A}
	\textbf{426} (1989), no.~1871, 211--232.
	\bibitem{CS1999}
	J.~H.~Conway, N.~J.~A.~Sloane,
	\emph{Sphere packings, lattices and groups: $3$rd Edition},
	Springer-Verlag,
	New York
	(1999).
		\bibitem{E}
	W.~Ebeling,
	\emph{Lattices and codes: A Course Partially Based on Lectures by F.~Hirzebruch, $3$rd Edition},
	Springer Spektrum,
	Berlin
	(2013).
	\bibitem{basic quadratic form}
	L.~J.~Gerstein,
	\emph{Basic quadratic forms},
	American Mathematical Society,
	Providence
	(2008).
%	\bibitem{canonical quadratic form} B.~W.~Jones, A canonical quadratic form for the ring of $2$-adic integers,
%	\emph{Duke Math. J.}
%	\textbf{11} (1944), 715--727.
	%	\bibitem{MAGMA}
%			W.~Bosma, J.~Cannon and C.~Playoust,
%			The Magma algebra system.~I.~The user language,
%			\emph{J.~Symbolic Comput.} \textbf{24} (1997), 235--265.
%		\bibitem{B}
%			R.~E.~Borcherds,
%			The Leech Lattice and Other Lattices,
%			 \emph{Ph.~D.~Dissertation, University of Cambridge} (1984).
		
	
	
%		\bibitem{CS1999}
%			J.~H.~ Conway, N.~J.~A.~ Sloane,
%			\emph{Sphere packings, lattices and groups: $3$rd Edition},
%			Springer-Verlag,
%			New York
%			(1999).
%		\bibitem{CS1982}
%			J.~H.~Conway, N.~J.~A.~Sloane,
%			The Unimodular Lattices of Dimension up to $23$ and the Minkowski-Siegel Mass Constants.
%			\emph{Eur. J. Comb.}
%			\textbf{3} (1982), 219--231.
	
%		\bibitem{KO1997}
%			M.~H.~Kim, B.~K.~Oh, Representations of Positive Definite Senary Integral Quadratic Forms by a Sum of Squares. \emph{J. Number Theory} \textbf{63} (1997), 89--100.
%		\bibitem{KO2005}
%			M.~H.~Kim, B.~K.~Oh, Representations of integral quadratic forms by sums of squares. \emph{Math. Zeitschrift} \textbf{250} (2005), 427--442.
		\bibitem{Ko1939}
		    C.~Ko,
		    On the decomposition of quadratic forms in six variables,
		    \emph{Acta Arith}. \textbf{3} (1939), 64--78.
		\bibitem{Ko1942a}
			C.~Ko,
		    On the decomposition of quadratic forms in seven variables,
		    \emph{Acad.~Sinica Sci.~Rec.} (1942), 30--33.
		\bibitem{Ko1942b}
			C.~Ko,
		    On the decomposition of quadratic forms in eight variables,
		    \emph{Acad.~Sinica Sci.~Rec.} (1942), 33--36.
%		\bibitem{M}
%			L.~J.~Mordell,
%			The representation of a definite quadratic form as a sum of two others.
%			\emph{Ann. Math. Second Ser.}
%			\textbf{38} (1937), 751--757.
		\bibitem{P}
			W.~Plesken,
			Additively Indecomposable Positive Integral Quadratic Forms.
			\emph{J. Number Theory}
			\textbf{47} (1994), 273--283.
%		\bibitem{PSW}
%			V.~Pless, N.~Sloane and H.~Ward,
%			Ternary codes of minimum weight $6$ and the classification of the self-dual codes of length $20$.
%			\emph{IEEE Trans. Inf. Theory}
%			 \textbf{26} (1980), 305--316.
%		\bibitem{SCIP}
%			A.~Gleixner, M.~Bastubbe, L.~Eifler, T.~Gally, G.~Gamrath, R.~L.~Gottwald, G.~Hendel,
%			C.~ Hojny, T.~Koch, M.~E.~L\"ubbecke, S.~J.~Maher, M.~Miltenberger, B.~M\"uller, M.~E.~Pfetsch,
%			C.~Puchert, D.~Rehfeldt, F.~Schl\"osser, C.~Schubert, F.~Serrano, Y.~Shinano,
%			J.~M.~Viernickel, M.~Walter, F.~Wegscheider, J.~T.~Witt, and J.~Witzig,
%			\emph{The SCIP	Optimization Suite 6.0}.
%			ZIB-Report 18-26,	%series
%			Zuse Institute Berlin (2018).
%		\bibitem{I1}
%			M.~I.~Icaza, Effectiveness in representations of positive definite quadratic forms. Dissertation. Ohio State Univ. (1992),
%		\bibitem{I2}
%			M.~I.~Icaza, Sums of squares of integral linear forms. \emph{Acta Arith.} \textbf{74} (1996), 231--240.
%		\bibitem{KimOh}
%			M.~Kim and B.~Oh, Representations of positive definite senary integral quadratic forms by a sum of squares. \emph{J. Number Theory} \textbf{63} (1997), 89--100.
		
	
	\end{thebibliography}
\end{document}